\documentclass[a4paper,11pt,reqno]{amsart}
\usepackage{a4wide}
\usepackage{import}

\usepackage{libertine}
\usepackage[libertine]{newtxmath}
\DeclareFontFamily{U}{mathx}{\hyphenchar\font45}
\DeclareFontShape{U}{mathx}{m}{n}{<->mathx10}{}
\DeclareFontSubstitution{U}{mathx}{m}{n}
\DeclareSymbolFont{mathx}{U}{mathx}{m}{n}

\DeclareMathSymbol{\intop}  {\mathop}{mathx}{"B3}
\DeclareMathSymbol{\iintop} {\mathop}{mathx}{"B4}
\DeclareMathSymbol{\iiintop}{\mathop}{mathx}{"B5}
\DeclareMathSymbol{\ointop} {\mathop}{mathx}{"B6}
\DeclareMathSymbol{\oiintop}{\mathop}{mathx}{"B7}

\usepackage{amsmath}
\usepackage{amsfonts}
\usepackage{amsthm}
\usepackage{graphicx}
\usepackage{color}
\usepackage{subfigure}
\usepackage{mathtools}
\usepackage{empheq}
\usepackage{cases}
\newcommand{\R}{\mathbb{R}}

\newcommand{\di}{{\rm div}}

\theoremstyle{plain}

\newtheorem{defi}{Definition}[section]
\newtheorem{thm}[defi]{Theorem}
\newtheorem{lem}[defi]{Lemma}
\newtheorem{rem}[defi]{Remark}

\newcommand{\eps}{\varepsilon}

\usepackage{geometry}
\author[M. Cirant]{Marco Cirant$^*$}
\address{$^*$ Dipartimento di matematica ``Tullio Levi-Civita'', Universit\`a  di Padova, Via Trieste 63, 35121, Italy}
\email{cirant@math.unipd.it}

\author[D. Ghilli]{Daria Ghilli$^\S$}
\address{$^\S$ Dipartimento di matematica ``Tullio Levi-Civita'', Universit\`a  di Padova, Via Trieste 63, 35121, Italy}
\email{ghilli@unipd.it}

\title[Existence and non-existence for time-dependent mean field games with strong aggregation]{Existence and non-existence for time-dependent mean field games with strong aggregation}

\begin{document}
\graphicspath{{./figures/}}

\date{\today}



\begin{abstract}
We investigate the existence of classical solutions to second-order quadratic Mean-Field Games systems with local and strongly decreasing couplings of the form $-\sigma m^\alpha$, $\alpha \ge 2/N$, where $m$ is the population density and $N$ is the dimension of the state space. We prove the existence of solutions under the assumption that $\sigma$ is small enough. For large $\sigma$, we show that existence may fail whenever the time horizon $T$ is large.

\end{abstract}

\maketitle
\date{\today}

\section{Introduction}\label{sec::Intro}

We consider in this paper systems of PDEs of the form
\begin{equation}\label{2}\tag{MFG}
\begin{cases}
	-u_t-\Delta u+ \frac{1}{2} |\nabla u|^2=-f(m)+ V(x)  & \mbox{ in } \R^N \times (0,T),\\
	m_t -\Delta m -\di(\nabla u m)=0 &  \mbox{ in } \R^N \times (0,T),\\
	m(0)=m_0, \quad u(T)=u_T & \mbox{ in } \R^N,
\end{cases}
\end{equation}
where $m_0$ is a smooth probability density, $u_T$ a smooth final cost, $V$ is a bounded potential, and $-f$ is a monotone non-increasing coupling. As a model problem, we consider
\begin{equation}\label{model}
-f(m) = -\sigma m^\alpha, \qquad \alpha \ge \frac2N, \qquad \sigma > 0.
\end{equation}
Such a system arises in the theory of Mean Field Games (MFG), a set of methods inspired by statistical physics to study Nash equilibria in (differential) games with a population of infinitely many identical players. The MFG toolbox has been introduced in the mathematical community by the seminal papers \cite{LL1,LL2,LL3}, and by a series of lectures at Coll\`ege de France by P.-L. Lions  \cite{PLLions}. A peculiarity of the present paper's setting, is that the coupling $m\mapsto-f(m)$ is assumed to have a {\it decreasing} character in $m$ (the minus sign in front of $f$ is to emphasize this fact). Since $-f(m)$ models the cost of a single agent in terms of the density $m$ of the population, \eqref{2} captures situations in which agents aim at maximizing aggregation. When the particular form \eqref{model} of $f$ is chosen, $\sigma$ and $\alpha$ are then related to the aggregation force. While \eqref{2} is known to enjoy uniqueness and long-time stability of solutions when $m\mapsto-f(m)$ is increasing (see e.g. \cite{PorrettaMinMax}  and references therein), the picture is less clear when $m\mapsto-f(m)$ is not increasing. Different phenomena have been observed in this framework, such as non-uniqueness of solutions \cite{BardiFischer17, MD}, periodic solutions \cite{CC, CNurb}, and instability in the long-time horizon \cite{Masoero19}. The main objective of this work is to investigate the existence of solutions when the coupling $-f$ has a strong decreasing character, that is when $\alpha$ in \eqref{model} satisfies
$$ \alpha \ge \frac2N.$$

The coefficient $\frac2N$ turns out to be crucial if one looks at the variational side of the problem. The system \eqref{2} is indeed known to be the optimality conditions of a minimization problem with PDE constraints (or Mean Field type optimal control problem). Global minimizers of this problem have been shown to exist only if $\alpha < \frac2N$ \cite{MD} (and these yield classical solutions \cite{CGoffi}); when $\alpha \ge \frac2N$, the variational problem is not even bounded from below. For this reason, the latter regime poses structural difficulties even for the {\it existence} of solutions to \eqref{2}. Solutions are indeed known to exist for general $\alpha > 0$ only when the time-horizon $T$ is small, by means of (non-variational) techniques involving perturbations of the heat equation, see \cite{Ambrose20, CGianniMannucci} (and references therein). We aim here at developing some new methods to explore existence without requiring $T$ to be ``small enough''.

\smallskip
A first main result of this paper is that for $T$ large, solutions to \eqref{2} {\it may not even exist}. Our main assumption on $f$, involving its anti-derivative $F(m):= \int_0^m f(s) \, ds$ also, reads as follows:
\begin{equation}\label{fass}
	\text{$f \in C^2((0,+\infty))$ \quad and \quad $N f(m)m-(N+2)F(m)\geq 0$}.
\end{equation}
Note that \eqref{fass} implies $F(m) \ge cm^{\frac{N+2}N}$ for some $c > 0$. When $f$ has the form \eqref{model}, then \eqref{fass} holds for all $\sigma > 0$. Regarding the initial/final data and the potential $V$, we assume that
	\begin{gather}
	\text{$V \in C^2_b(\R^N),$ \quad and \quad $2(V- \inf_{\R^N}V) + \nabla V \cdot x \ge 0$ on $\R^N$}, \label{Vass} \\
	\text{$u_T \in C^4_b(\R^N),$ \quad and \quad $ \nabla u_T \cdot x\geq 0$}, \label{dataass} \\
	\text{$m_0 \in C^4_b(\R^N),$ \, $m_0, |x| m_0 , x^2 m_0, \nabla m_0 \in L^1(\R^N)$ \quad and \quad $\textstyle\int_{\R^N} m_0 \, dx =1$, $m_0 \ge 0$}. \label{dataass2}
	\end{gather}
Note that the condition on $V$ is not much restrictive, and allows even for radially decreasing potentials (up to some degree). Then we show that, if an additional condition involving $m_0$, $f$, $V$ is satisfied, then \eqref{2} has no solutions if $T$ is {\it large}.
\begin{thm}\label{thm:noexist} Assume that \eqref{Vass}, \eqref{dataass}, \eqref{dataass2} and \eqref{fass} holds. Suppose that
	\begin{equation}\label{e0ass}
	e_0 := -\frac{1}{2}\int_{\R^N}\frac{|\nabla m_0|^2}{m_0}\, dx+\int_{\R^N} F(m_0)\, dx-\int_{\R^N}(V- \inf_{\R^N}V)m_0 \, dx > 0.
	\end{equation}
	Then, if
	\begin{equation*}
	T > \frac{N}{2e_0} + \sqrt{\frac{\int_{\R^N}x^2m_0 \, dx}{2e_0}},
	\end{equation*}
	the system \eqref{2} has no classical solutions.
	\end{thm}
Let us stress that the condition $e_0 > 0$ may be realized or not depending on $m_0$ and $f$ (and the oscillation of $V$). When $f(m) = \sigma m^{\alpha}$, note that for any fixed $m_0$, replacing it by $\epsilon^{-N} m_0(\epsilon^{-1} x)$ into \eqref{e0ass} yields
\[
\epsilon^{-N} e_0 = - \frac{\epsilon^{-(N+2)}}{2}\int_{\R^N}\frac{|\nabla m_0|^2}{m_0}\, dx+ \frac{\sigma\epsilon^{-N(\alpha+1)}}{\alpha+1} \int_{\R^N} m_0^{\alpha+1}\, dx-\epsilon^{-N}\int_{\R^N}(V(\epsilon x)- \inf_{\R^N}V)m_0(x)dx,
\]
hence $e_0 > 0$ when the second term in the right-hans side is dominating, that is when $\epsilon$ is small enough or $\sigma$ is large enough. In other words, non-existence is triggered by ``concentration'' of the initial datum, or ``strength'' of the aggregation force. The proof of Theorem \ref{thm:noexist} involves the study of the evolution of second order moments $h(t) = \int x^2 m(x,t)\, dx$. The core identity in Lemma \ref{lem:conserv} shows that under the standing assumptions, $h$ has to be strictly convex, but given the information at $t = 0, t= T$, this forces $h$ to be negative when $T$ is large, which is impossible. Lemma \ref{lem:conserv} is based on two structural estimates: the first one is the well-known conservation of energy (a quantity which stems from the Hamiltonian nature of \eqref{2}). The second one is a new identity which is obtained by testing the equations by projections of $\nabla m, \nabla u$ over the direction $x$, and is some sense related to dilations properties of the variational problem. We mention that a similar approach was used to obtain non-existence in \cite{CirCPDE16} for stationary problems (for $\alpha > \frac{2}{N-2}$), but the analysis developed here for the evolutive case is more involved (and heavily related to the quadratic dependance with respect to the gradient in the first equation, see Remark \ref{nonquad}).

By the very same procedure, we obtain also non existence results for the so-called Planning Problem in MFG. In such a framework, one wants to drive agents from an initial configuration $m_0$ to a final one $m_T$, optimizing some cost. This problem is related to a PDE system of the form \eqref{2}, where there is no fixed final condition $u_T$ for $u$, but rather a final condition $m(T) = m_T$. Our results on the planning problem are described in Section \ref{splanning}.
\smallskip

Theorem \ref{thm:noexist} leaves open the question, for a fixed $m_0$, of the existence of solutions to \eqref{2} when $f$ is ``small'' (that is for small $\sigma$ in the model case). Let us then describe the second main result of this paper. Assume that for some $\sigma > 0$,
\begin{equation}\label{eq:ipofex}
f \in C^2(\R^+), \, f(0)=0, \, f \ge 0, \quad | f'(m)|\leq \sigma\alpha m^{\alpha-1}, \quad \alpha < 
\begin{cases}
+\infty & N=1,2 \\
\frac2{N-2} & N \ge 3.
\end{cases}
\end{equation}
Note that we are not requiring $f$ to be increasing, but rather that $f$ grows at most like $\sigma m^\alpha$. The model case \eqref{model} perfectly falls into this setting. We also suppose that
\begin{equation}\label{eq:ipodex}
\begin{gathered}
	\text{$V \in C^2_b(\R^N)$, $m_0, u_T \in C^4_b(\R^N)$ },  \\
	\text{$m_0, |x| m_0  \in L^1(\R^N)$ \quad and \quad $\textstyle \int_{\R^N} m_0 \, dx =1$, $m_0 \ge 0$.}
\end{gathered}
\end{equation}
Then, we prove existence of solutions for $\sigma$ small.
\begin{thm}\label{thm:exvmu}
	Assume \eqref{eq:ipofex} and \eqref{eq:ipodex}. Then, there exists $\sigma_0 > 0$ depending on $N$, $\alpha$, $\|m_0\|_{L^{\alpha+1}(\R^N)}$, $\|u_T\|_{C^2(\R^N)}, T||\Delta V||_{L^\infty(\R^N)}$, such that for any $$\sigma \le \sigma_0$$($\sigma$ appearing in \eqref{eq:ipofex}), the system \eqref{2} has a classical solution $(u, m)$. Note that if $\Delta V=0$, then $\sigma_0$ is independent of $T$. 
	
	\end{thm}
	
We stress that $\sigma_0$ is affected by $T$ only when the potential $V$ is non-trivial. When there is no spatial potential, i.e. $V \equiv 0$, and $f$ is fixed, solutions are proven to exist for all $T$ (and will probably ``disappear'' as $T \to \infty$, see Remark \ref{longtime}). Note also that we require rather smooth initial/final data, but existence restrictions depend only on $\|m_0\|_{L^{\alpha+1}(\R^N)}$ and $\|u_T\|_{C^2(\R^N)}$ (thus allowing to relax the smoothness assumptions via approximation arguments). As we previously observed, the only known existence results require $T$ small, and approach \eqref{2} as a perturbation of two heat equations; due to the presence of the non-linear term $|\nabla u|^2$, this strategy does not allow to analyze the ``small'' $\sigma$ regime. The key step here is an a priori estimate which is obtained heavily relying again on the MFG structure. We use a combination of the conservation of energy, so-called second-order estimates, and parabolic regularization to get an inequality of the form
$$
	\left(\int_0^T\int_{\R^N}m^{2\alpha+1}\right)^{\beta}\lesssim \sigma^2 \int_0^T\int_{\R^N}m^{2\alpha+1} +T||\Delta V||_{L^\infty(\R^N)} + 1\qquad \beta \in (0,1) .
$$
Note that in view of its super-linear nature, the previous estimate is meaningful only for $\sigma$ small. In that case it is possible to set up a Schaefer's fixed point procedure revolving around the boundedness of $\iint m^{2\alpha+1}$. Since $\alpha < \frac2{N-2}$, this yields boundedness of $f(m)$ in $L^\frac{N+2}2$, which is enough to set up a bootstrap procedure.
We point out that the restriction $\frac2{N-2}$ on $\alpha$ might be structural; we do not know at this stage how to construct solutions for $\alpha \ge \frac2{N-2}$ and arbitrary $T$.

We finally mention that our existence scheme does not seem to apply easily to the Planning Problem. To our knowledge, when $-f$ is not increasing, existence is an open problem even in the short-time horizon regime.
\smallskip

{\bf Existence versus non-existence.} For the sake of clarity, we summarize below, for fixed initial-final data $m_0, u_T$ and potential $V$, existence and non-existence regimes as $\sigma$ and $T$ vary. These are sketched in Figure \ref{graf}, for $V \equiv 0$ and $V \neq 0$ (we again consider the model coupling \eqref{model}).

First, we note that for any $\sigma > 0$, there exists $\overline T = \overline T(\sigma)$ such that \eqref{2} has solutions provided that $T \le \overline T$. Though this existence result is not stated explicitly anywhere, it can be derived via a straightforward adaptation of \cite[Theorem 1.4]{MD} from the flat torus to the euclidean setting $\R^N$. This (standard) short-time existence situation is light-blue coloured in Figure \ref{graf}.

Regarding our existence theorem, it says that there exists $\sigma_0 = \sigma_0(T)$ such that \eqref{2} has solutions for all $\sigma \le \sigma_0$. If $V \equiv 0$, $\sigma_0$ is proven to be $T$-independent, as shown in Figure \ref{graf} (green region).

Finally, we prove that \eqref{2} has {\it no} solutions whenever
$$C_1 \sigma - C_2 > 0, \quad \text{where $C_1 = \frac{1}{\alpha+1}\int_{\R^N} m_0^{\alpha+1}, \, C_2 = \frac{1}{2}\int_{\R^N}\frac{|\nabla m_0|^2}{m_0} + \int_{\R^N}(V- \inf_{\R^N}V)m_0$},$$
and
$$ T > \frac{N}{2(C_1 \sigma - C_2)} +  \sqrt{\frac{\int_{\R^N}x^2m_0 \, dx}{2(C_1 \sigma - C_2)}}. $$
Equivalently, for all $T > 0$ there exists $\sigma_* = \sigma_*(T)$ such that \eqref{2} has no solutions for any $\sigma > \sigma_*$. This is the orange region in Figure \ref{graf}. Note that the upper bound $\sigma_*(T)$ on $\sigma$  for which existence is expected goes to $+\infty$ as $T\to 0$, while $\sigma_*(T) \to C_2 C_1^{-1}$ as $T \to \infty$. Unfortunately, $\sigma_0$ and $C_2 C_1^{-1}$ are obtained in completely different ways here, and in general they do not coincide (though the non-existence condition can be in some sense optimized, see Remark \ref{Timprove}).

\begin{figure}
\centering
\def\svgwidth{.48\columnwidth}
\begingroup%
  \makeatletter%
  \providecommand\color[2][]{%
    \errmessage{(Inkscape) Color is used for the text in Inkscape, but the package 'color.sty' is not loaded}%
    \renewcommand\color[2][]{}%
  }%
  \providecommand\transparent[1]{%
    \errmessage{(Inkscape) Transparency is used (non-zero) for the text in Inkscape, but the package 'transparent.sty' is not loaded}%
    \renewcommand\transparent[1]{}%
  }%
  \providecommand\rotatebox[2]{#2}%
  \newcommand*\fsize{\dimexpr\f@size pt\relax}%
  \newcommand*\lineheight[1]{\fontsize{\fsize}{#1\fsize}\selectfont}%
  \ifx\svgwidth\undefined%
    \setlength{\unitlength}{434.693121bp}%
    \ifx\svgscale\undefined%
      \relax%
    \else%
      \setlength{\unitlength}{\unitlength * \real{\svgscale}}%
    \fi%
  \else%
    \setlength{\unitlength}{\svgwidth}%
  \fi%
  \global\let\svgwidth\undefined%
  \global\let\svgscale\undefined%
  \makeatother%
  \begin{picture}(1,0.69972897)%
    \lineheight{1}%
    \setlength\tabcolsep{0pt}%
    \put(0,0){\includegraphics[width=\unitlength,page=1]{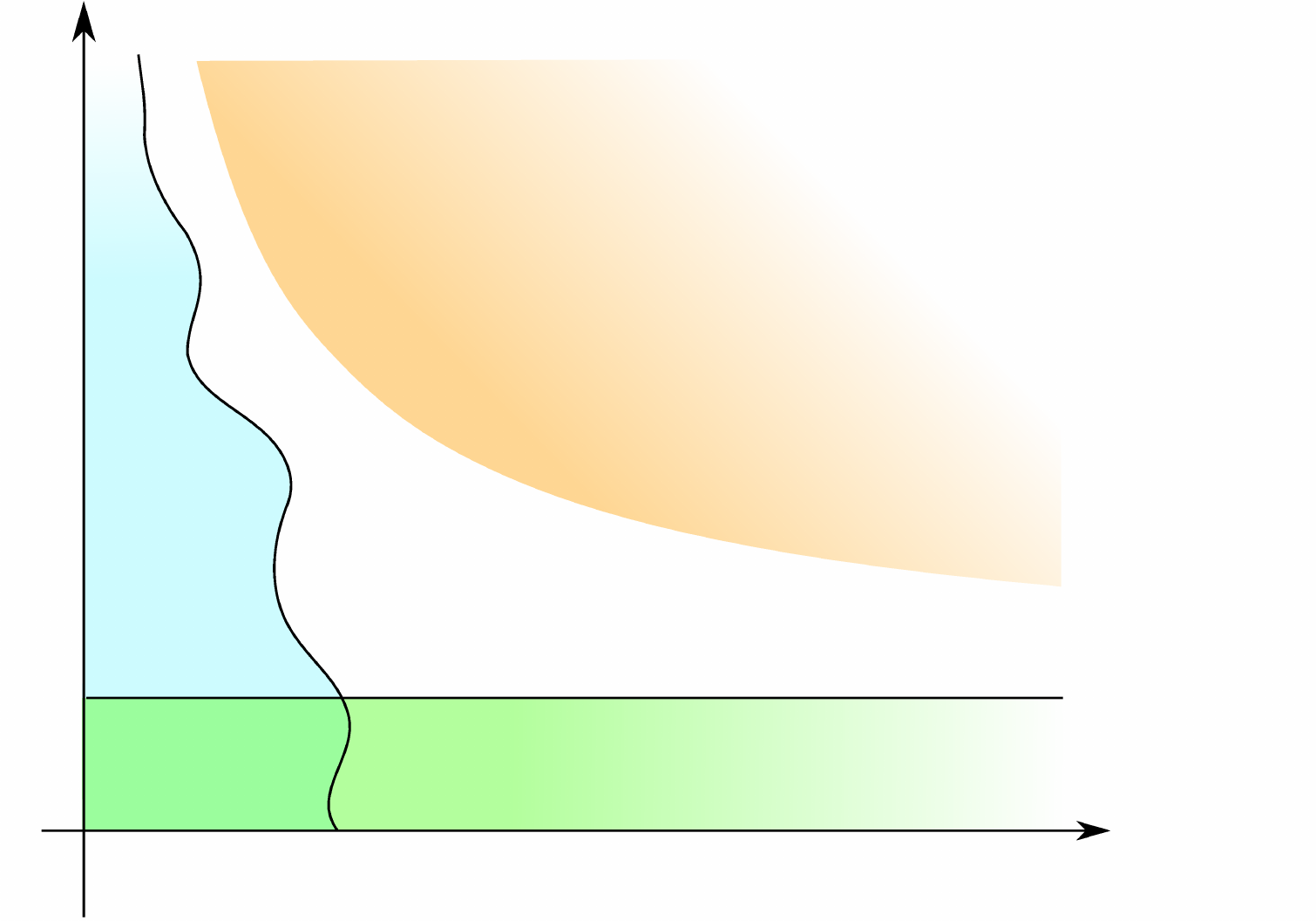}}%
    \put(0.79530946,0.01171754){\color[rgb]{0,0,0}\makebox(0,0)[lt]{\lineheight{1.25}\smash{\begin{tabular}[t]{l}$T$\end{tabular}}}}%
    \put(0.01171769,0.62594013){\color[rgb]{0,0,0}\makebox(0,0)[lt]{\lineheight{1.25}\smash{\begin{tabular}[t]{l}$\sigma$\end{tabular}}}}%
    \put(0,0){\includegraphics[width=\unitlength,page=2]{grafico.pdf}}%
    \put(0.83412473,0.24538427){\color[rgb]{0,0,0}\makebox(0,0)[lt]{\lineheight{1.25}\smash{\begin{tabular}[t]{l}$\sigma_*(T)$\end{tabular}}}}%
    \put(0.83575817,0.15946424){\color[rgb]{0,0,0}\makebox(0,0)[lt]{\lineheight{1.25}\smash{\begin{tabular}[t]{l}$\sigma_0$\end{tabular}}}}%
    \put(0.23007511,0.24731492){\color[rgb]{0,0,0}\makebox(0,0)[lt]{\lineheight{1.25}\smash{\begin{tabular}[t]{l}$\overline{T}(\sigma)$\end{tabular}}}}%
    \put(0,0){\includegraphics[width=\unitlength,page=3]{grafico.pdf}}%
  \end{picture}%
\endgroup%

\def\svgwidth{.48\columnwidth}
\begingroup%
  \makeatletter%
  \providecommand\color[2][]{%
    \errmessage{(Inkscape) Color is used for the text in Inkscape, but the package 'color.sty' is not loaded}%
    \renewcommand\color[2][]{}%
  }%
  \providecommand\transparent[1]{%
    \errmessage{(Inkscape) Transparency is used (non-zero) for the text in Inkscape, but the package 'transparent.sty' is not loaded}%
    \renewcommand\transparent[1]{}%
  }%
  \providecommand\rotatebox[2]{#2}%
  \newcommand*\fsize{\dimexpr\f@size pt\relax}%
  \newcommand*\lineheight[1]{\fontsize{\fsize}{#1\fsize}\selectfont}%
  \ifx\svgwidth\undefined%
    \setlength{\unitlength}{432.40320689bp}%
    \ifx\svgscale\undefined%
      \relax%
    \else%
      \setlength{\unitlength}{\unitlength * \real{\svgscale}}%
    \fi%
  \else%
    \setlength{\unitlength}{\svgwidth}%
  \fi%
  \global\let\svgwidth\undefined%
  \global\let\svgscale\undefined%
  \makeatother%
  \begin{picture}(1,0.70873036)%
    \lineheight{1}%
    \setlength\tabcolsep{0pt}%
    \put(0,0){\includegraphics[width=\unitlength,page=1]{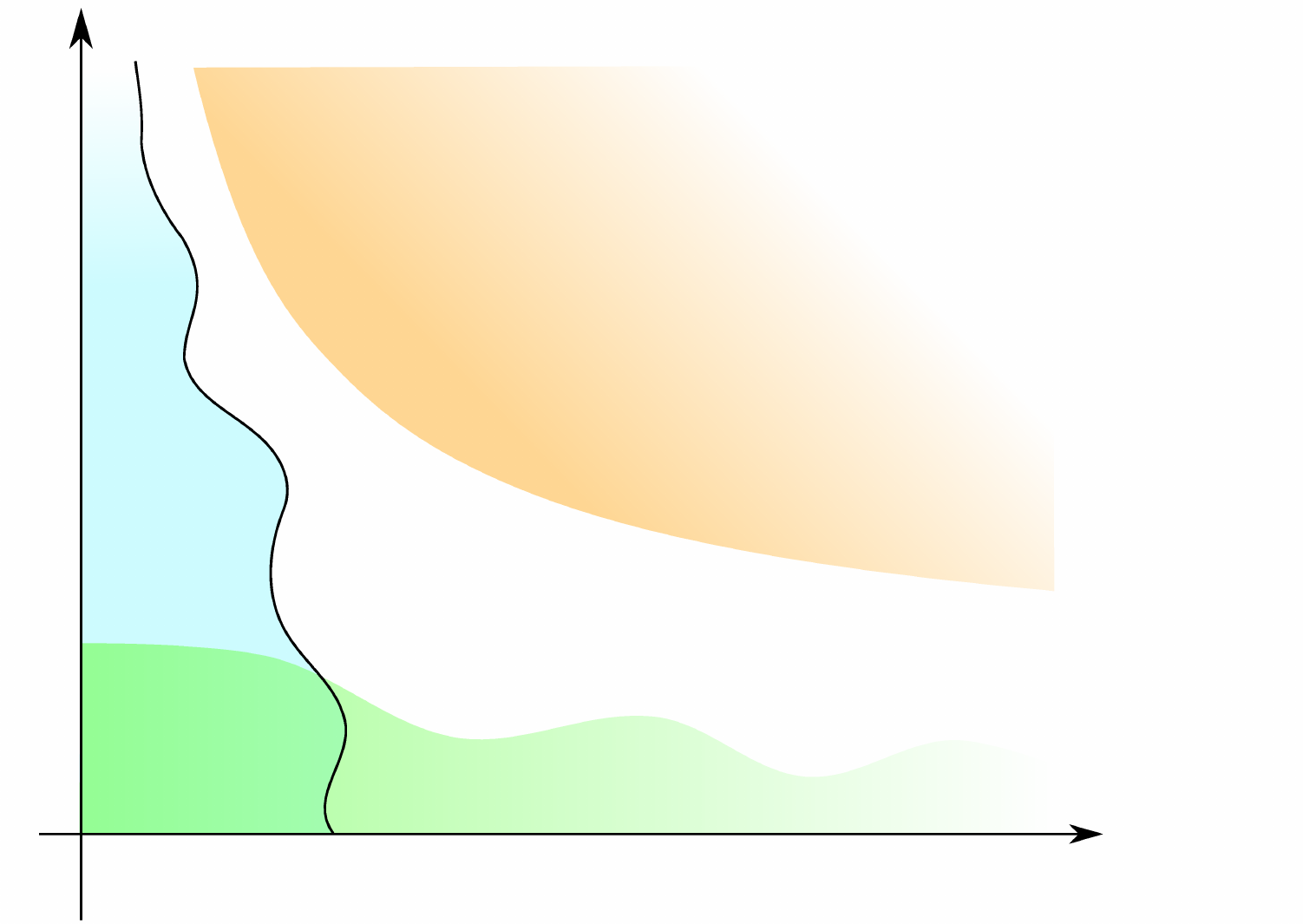}}%
    \put(0.79775598,0.01177959){\color[rgb]{0,0,0}\makebox(0,0)[lt]{\lineheight{1.25}\smash{\begin{tabular}[t]{l}$T$\end{tabular}}}}%
    \put(0.00654549,0.62925497){\color[rgb]{0,0,0}\makebox(0,0)[lt]{\lineheight{1.25}\smash{\begin{tabular}[t]{l}$\sigma$\end{tabular}}}}%
    \put(0,0){\includegraphics[width=\unitlength,page=2]{graficoV.pdf}}%
    \put(0.8367768,0.24668377){\color[rgb]{0,0,0}\makebox(0,0)[lt]{\lineheight{1.25}\smash{\begin{tabular}[t]{l}$\sigma_*(T)$\end{tabular}}}}%
    \put(0.83841889,0.11521254){\color[rgb]{0,0,0}\makebox(0,0)[lt]{\lineheight{1.25}\smash{\begin{tabular}[t]{l}$\sigma_0(T)$\end{tabular}}}}%
    \put(0,0){\includegraphics[width=\unitlength,page=3]{graficoV.pdf}}%
    \put(0.23457302,0.25058384){\color[rgb]{0,0,0}\makebox(0,0)[lt]{\lineheight{1.25}\smash{\begin{tabular}[t]{l}$\overline{T}(\sigma)$\end{tabular}}}}%
  \end{picture}%
\endgroup%

\caption{\footnotesize Green and light blue are existence regions, while the orange one is the non-existence region as the coupling strength $\sigma$ and the time horizon $T$ vary. On the left $V \equiv 0$, while on the right $V \neq 0$. }\label{graf}
\end{figure}

\bigskip

{\bf Acknowledgements.} The authors are members of the Gruppo Nazionale per l'Analisi Matematica, la Probabilit\`a e le loro Applicazioni (GNAMPA) of the Istituto Nazionale di Alta Matematica (INdAM). They are partially supported by the research project ``Nonlinear Partial Differential Equations: Asymptotic Problems and Mean-Field Games" of the Fondazione CaRiPaRo.

\bigskip

{\bf Notations. } We will denote by $C^k_b(\Omega)$, $\Omega \subseteq \R^M$, the space of bounded continuous functions with bounded continuous derivatives up to order $k \in \mathbb N$ (if $k=0$, then $k$ is omitted). For $a > 0$, $C^{2a, a}$ will be the standard parabolic H\"older space. $C([0,T]; X)$ will denote the space of continuous functions with values in a Banach space $X$. Finally, the standard $L^2$ parabolic energy space will be $V_{2}(\Omega \times (0,T)) = L^\infty((0,T); L^2(\Omega))\cap L^2(0,T; W^{1,2}(\Omega))$, and $m \in V_{2, {\rm loc}}(\R^N \times (0,T))$ will mean that $m \in V_{2}(\Omega' \times (0,T))$ for all $\Omega'$ with compact closure in $\Omega$.

\bigskip

	\section{Non-existence}
	
	This section is devoted to the proof of Theorem \ref{thm:noexist}, which will be based on several lemmas. Before we start, let us comment on the assumptions on $u, m$ we will work with.
	
\begin{rem} \upshape Throughout the section we will assume that $u_t, \nabla u$, $\nabla u_t, \Delta u$, $m, m_t$, $\nabla m, \nabla m_t$, $ \Delta m$ belong to $C_b(\R^N \times [0,T])$. This degree of regularity is coherent with the one coming from the existence theorem that will be proven in the next section. Such a regularity can be obtained starting from any classical solution $(u,m) \in C^{2,1}$, by means of parabolic Schauder estimates (and the standing assumptions on $m_0, u_T, f, V$). All the arguments below actually need only polynomial growth in the $x$-variable for $u, m$ and their derivatives, and that $|Du|^2m, F(m) \in L^1$.

We stress that the assumption that $(u,m)$ is a {\it classical} solution is not really crucial to get non-existence. For example, arguing as in the proof of existence, a bootstrap procedure shows that weak solutions in a suitable (energy) sense have to be smooth, and thus Theorem \ref{thm:noexist} applies. In other words, one can formulate the same non-existence result for a large class of weak solutions.
\end{rem}

 First, we show that if $m$ is bounded and smooth, it has to be a continuous flow of probability densities.

\begin{lem}\label{lem:intmnablam}
	Let $m$ be a non-negative classical solution to (the second equation in) \eqref{2} such that $m, m_t, \nabla m, \Delta m, \nabla u \in C_b(\R^N \times (0,T))$, and $m_0 \in L^1(\R^N)$. Then $m$ is non-negative, $m \in C([0,T]; L^1(\R^N))$, and $\|m(t)\|_{L^1(\R^N)} = \|m_0\|_{L^1(\R^N)}$ for all $t \in [0,T]$. 
\end{lem}
\begin{proof}  First, $m$ is non-negative by the maximum principle. 

To control the $L^1$-norm of $m(t)$, we multiply second equation in \eqref{2} by $\phi_\epsilon:=e^{-\epsilon \sqrt{|x|^2+1}}$, $\epsilon > 0$, and integrate over $(0,t) \times \R^N$ to obtain
		\begin{equation*}
		\int_0^t \int_{\R^N}m_t \phi_\epsilon\, dx dt=\int_0^t \int_{\R^N} \Delta m \phi_\epsilon\, dx dt+\int_0^t \int_{\R^N}\di(\nabla um)\phi_\epsilon\, dx dt.
		\end{equation*}
		Since $m, m_t, \nabla m, \Delta m, \nabla u$ are bounded, and $\phi_\epsilon, \nabla \phi_\epsilon, \Delta \phi_\epsilon$ are in $L^1(\R^N)$, we can use Lemma \ref{lem:intp} to integrate by parts, to get
		\begin{multline}\label{eq000}
		 \int_{\R^N}m(t)\phi_\epsilon\, dx -  \int_{\R^N}m(0)\phi_\epsilon\, dx  \\ = \int_{0}^{t}\partial_t \int_{\R^N}m\phi_\epsilon\, dx=\int_0^t \int_{\R^N} m  \Delta \phi_\epsilon\, dx dt-\int_0^t \int_{\R^N}\nabla um \cdot \nabla \phi_\epsilon\, dx dt.
		\end{multline}
		Therefore,
		$$
		\int_{\R^N}m(t)\phi_\epsilon\, dx  \le  \int_{\R^N}m(0)\phi_\epsilon\, dx + K_\eps\int_0^t \int_{\R^N} m\phi_\epsilon \, dx dt,
		$$
		where $K_\eps= \sup_{\R^N \times [0,T]} (|\Delta \phi_\epsilon| + |\nabla u| |\nabla \phi_\epsilon| ) \phi_\epsilon^{-1} $. By Gronwall's Lemma, $$ \int_{\R^N}m(t)\phi_\epsilon\, dx \le e^{K_\eps t} \int_{\R^N}m_0\phi_\epsilon\, dx \quad \forall t \in [0,T]. $$
		By a straightforward computation, $K_\eps \to 0$, $\eps \to 0$, and $\phi_\epsilon \nearrow 1$, hence the Monotone Convergence Theorem yields $\|m(t)\|_{L^1(\R^N)} \le \|m(0)\|_{L^1(\R^N)}$ for all $t$. To prove the equality, it is sufficient to go back to \eqref{eq000}, and pass to the limit $\eps \to 0$. Indeed, now we now that $\sup_{[0,T]} \|m(t)\|_{L^1(\R^N)}$ is bounded, and $\Delta \phi_\epsilon$, $\nabla \phi_\epsilon$ converge uniformly to zero as as $\eps \to 0$ . Finally, since $\|m(t)\|_{L^1(\R^N)} = \|m_0\|_{L^1(\R^N)}$ for all $t$, and $m(t) \to m(t_0)$ a.e. on $\R^N$ as $t \to t_0$ (for any fixed $t_0$), one can conclude $m \in C([0,T]; L^1(\R^N))$.
		
\end{proof}

The following lemma describes the evolution of second order moments of $m$, and concerns integrability properties of $\frac{|\nabla m|^2}{m}$ and the crossed quantity $ |\nabla u| |\nabla m|$.

\begin{lem}\label{lem:ddtmoment}
	In addition to the assumptions of previous Lemma \ref{lem:intmnablam}, suppose that $|x|m_0, x^2 m_0 \in L^1(\R^N)$.
	Then, 
	\begin{itemize}
	\item[{\it (i)}] $t \mapsto \int_{\R^N} x^2m(t) \, dx$ is Lipschitz continuous, and
	$$\frac d{dt}\int_{\R^N} m(t) x^2\, dx = 2N \int_{\R^N} m_0(x)dx - \int_{\R^N} m(t) \nabla u(t) \cdot x\, dx \quad \text{for a.e. $t$}.$$
	\item[{\it (ii)}] $\frac{|\nabla m|^2}{m}, \ |\nabla u| |\nabla m| \in L^1(\R^N \times (0,T))$.
	\end{itemize}
\end{lem}

\begin{proof}

To prove {\it (i)}, we employ Lemma \ref{lem:mombounds}. First, $m \in C([0,T]; L^1(\R^N))$ by Lemma \ref{lem:intmnablam}, so $|\nabla u|m \in L^1(\R^N \times (0,T)$. By Lemma \ref{lem:mombounds} {\it (i)} we get $\sup_{t \in [0,T]} \int_{\R^N} |x| m \, dx< \infty$, and therefore $|\nabla u \cdot x|m \in L^\infty((0,T); L^1(\R^N) )$. By Lemma \ref{lem:mombounds} {\it (ii)}, we obtain $$\int_{\R^N}  x^2m(t)\, dx = \int_{\R^N}  x^2m_0 \, dx + 2Nt \int_{\R^N} m_0(x)dx - \int_0^t \int_{\R^N} m\nabla u \cdot x \, dxdt, $$
which is the desired statement.

Item {\it (ii)} follows by results in \cite{BRS}. Indeed, in view of our standing assumptions and previous results, $|\nabla u|^2 m, (\ln \max(|x|,1))^2m \in L^1(\R^N \times (0,T))$ and $m_0 \max(0, \ln m_0) \in L^1(\R^N)$. Therefore, by Theorem 2.1 and Remark 2.1 in \cite{BRS}, $\frac{|\nabla m|^2}{m} \in L^1(\R^N \times (0,T)$. Finally, H\"older's inequality yields
$$
\int_0^T \int_{\R^N}  |\nabla u| |\nabla m| \, dx dt \le \left(\int_0^T \int_{\R^N}  |\nabla u|^2 m \, dx dt \right)^{\frac12} \left(\int_0^T \int_{\R^N} \frac{|\nabla m|^2}{m} \, dx dt  \right)^{\frac12} < \infty.
$$
\end{proof}

	The following lemma shows that some quantity involving $\nabla u,m$ is conserved in time. This conserved quantity is related to the Hamiltonian nature of the MFG system (see \cite{CNurb} for further comments about that).

				\begin{lem}\label{lem:E} Let $(u, m)$ be a classical solution of \eqref{2}. In addition to the assumptions of previous Lemma \ref{lem:intmnablam}, suppose that 
				$u_t, \nabla u_t, \Delta u, \nabla m_t\in C_b(\R^N \times [0,T])$.
				Then, the following statements hold:
			\begin{itemize}
				\item[{\it (i)}] $\int_{\R^N} \nabla u \cdot \nabla m \, dx+\frac{1}{2}\int_{\R^N} |\nabla u|^2m\, dx+\int_{\R^N} F(m) \, dx-\int_{\R^N}Vm \, dx = E \in \mathbb R \quad \text{for a.e. $t$}$
				\item[{\it (ii)}] assuming in addition that $\nabla m_0 \in L^1(\R^N)$,	$$
				E\geq -\frac{1}{2}\int_{\R^N} \frac{|\nabla m_0|^2}{m_0}\, dx+\int_{\R^N} F(m_0)\, dx-\int_{\R^N} Vm_0\, dx.
				$$
			\end{itemize}
		\end{lem}
		
		\begin{proof}
		We start with claim {\it (i)}. The (standard) idea is to multiply the second equation in \eqref{2} by $u_t$ and the first one by $m_t$, and perform several integration by parts. Since nothing is assumed regarding the integrability of $m_t, \nabla m, \ldots$ on $\R^N$, we multiply the second equation in \eqref{2} by $u_t\phi_\epsilon$ and the first one by $m_t\phi_\epsilon $, where $\phi_\epsilon=e^{-\epsilon\frac{|x|^2}{2}}$, $\epsilon > 0$, and integrate over $\R^N$ to get
		\begin{equation}\label{lem11}
		\int_{\R^N}[-\Delta u m_t\phi_\epsilon -\Delta m u_t\phi_\epsilon+\frac{1}{2}|\nabla u|^2 m_t\phi_\epsilon-\mbox{div}(\nabla u m) u_t \phi_\epsilon+f(m)m_t\phi_\epsilon -Vm_t\phi_\epsilon]\, dx=0.
		\end{equation}
		Since $\phi_\epsilon, \nabla \phi_\epsilon, \Delta\phi_\epsilon \in L^1(\R^N)$, by the boundedness in $\sup$-norm of $u_t, \nabla u$, $\nabla u_t, \Delta u$, $m, m_t$, $\nabla m, \nabla m_t$, $ \Delta m$ over $\R^N \times [0,T]$, all the integrations by parts (using Lemma \ref{lem:intp} below) are justified. Then, we will let $\epsilon \to 0$.
		
		We start with the first two terms in \eqref{lem11}, integrating repeatedly by parts to obtain
		\begin{align}\label{lem112}
		\int_{\R^N}-\Delta u m_t\phi_\epsilon -\Delta m u_t\phi_\epsilon \, dx & = 
		 \int_{\R^N} \nabla u \nabla m_t \phi_\epsilon + \nabla u m_t \nabla \phi_\epsilon + \nabla m \nabla u_t \phi_\epsilon + \nabla m u_t \nabla \phi_\epsilon \, dx \\
		 & =  \int_{\R^N} \partial_t \big(\nabla u \nabla m \phi_\epsilon\big) + \nabla u m_t \nabla \phi_\epsilon -m \nabla u_t \nabla \phi_\epsilon - m u_t \Delta \phi_\epsilon\, dx \nonumber \\
		 & =  \int_{\R^N} \partial_t \big(\nabla u \nabla m \phi_\epsilon\big) + \partial_t \big( \nabla u m \nabla \phi_\epsilon\big) -2 m \nabla u_t \nabla \phi_\epsilon - m u_t \Delta \phi_\epsilon\, dx. \nonumber
		\end{align}
		Then,
		\begin{align}\label{lem113}
		\int_{\R^N} \frac{1}{2}|\nabla u|^2 m_t\phi_\epsilon-\di(\nabla u m) u_t \phi_\epsilon & = 
		\int_{\R^N} \frac{1}{2}|\nabla u|^2 m_t\phi_\epsilon + \nabla u m \nabla u_t \phi_\epsilon + \nabla u  m \, u_t \nabla \phi_\epsilon \, dx
		\\ & = \int_{\R^N} \partial_t \Big(\frac{1}{2}|\nabla u|^2 m\phi_\epsilon \Big) + \nabla u m \, u_t \nabla \phi_\epsilon \, dx. \nonumber
		\end{align}
		Finally,
		\begin{equation}\label{lem114}
		\int_{\R^N} f(m)m_t\phi_\epsilon -Vm_t\phi_\epsilon\, dx=\int_{\R^N} \partial_t \Big(F(m)\phi_\epsilon - Vm \phi_\epsilon \Big)dx.
		\end{equation}
		
		Thus, plugging \eqref{lem112}, \eqref{lem113} and \eqref{lem114} into \eqref{lem11} yields
		\begin{multline*}
		\int_{\R^N}  \partial_t \big(\nabla u \nabla m \phi_\epsilon + \frac{1}{2}|\nabla u|^2 m\phi_\epsilon + F(m)\phi_\epsilon - Vm \phi_\epsilon + \nabla u m \nabla \phi_\epsilon\big) \, dx = \\
		\int_{\R^N} 2 m \nabla u_t \nabla \phi_\epsilon + m u_t \Delta \phi_\epsilon  - \nabla u m \, u_t \nabla \phi_\epsilon \, dx.
		\end{multline*}
		Again by the presence of  $\phi_\epsilon, \nabla \phi_\epsilon \in L^1(\R^N)$, and boundedness of $u,m$ and their derivatives, we have
		\begin{multline*}
		\partial_t \int_{\R^N}  \nabla u \nabla m \phi_\epsilon + \frac{1}{2}|\nabla u|^2 m\phi_\epsilon + F(m)\phi_\epsilon - Vm \phi_\epsilon + \nabla u m \nabla \phi_\epsilon \, dx = \\
		\int_{\R^N} \big(2 \nabla u_t \nabla \phi_\epsilon + u_t \Delta \phi_\epsilon  - \nabla u \, u_t \nabla \phi_\epsilon) m \, dx,
		\end{multline*}
		and for all $t_1 \le t_2$,
		\begin{multline*}
		\int_{\R^N}  \nabla u \nabla m \phi_\epsilon + \frac{1}{2}|\nabla u|^2 m\phi_\epsilon + F(m)\phi_\epsilon - Vm \phi_\epsilon + \nabla u m \nabla \phi_\epsilon \, dx \ \Big|_{t=t_1}^{t=t_2} = \\
		\int_{t_1}^{t_2}\int_{\R^N} \big(2 \nabla u_t \nabla \phi_\epsilon + u_t \Delta \phi_\epsilon  - \nabla u \, u_t \nabla \phi_\epsilon) m \, dx dt.
		\end{multline*}
		
		Note now that $\int_{\R^N} m(t) \,dx= 1$ for all $t$, and $m \in C_b(\R^N \times [0,T])$. Therefore, $\nabla u_t  m$, $ u_t  m$,  $\nabla u \, u_t  m \in L^1(\R^N \times (0,T))$, and $|\nabla u|^2 m(t), F(m(t)), Vm(t) \in L^1(\R^N)$ for all $t$. Furthermore, $\nabla u \nabla m (t) \in L^1(\R^N)$ for a.e. $t \in (0,T)$ by Lemma \ref{lem:ddtmoment}. Then, since $\phi_\epsilon \to 1$ and $\nabla \phi_\epsilon, \Delta\phi_\epsilon \to 0$ uniformly on $\R^N$ as $\epsilon \to 0$, by the Dominated Convergence Theorem one obtains
		\[
		\int_{\R^N}  \nabla u \nabla m + \frac{1}{2}|\nabla u|^2 m + F(m) - Vm  \, dx \ \Big|_{t=t_1}^{t=t_2} = 0.
		\]
		for a.e. $t_1, t_2 \in (0,T)$. Then, there exists $E \in \R$ such that for a.e. $t$
		\[
		\int_{\R^N}  \nabla u \nabla m(t) + \frac{1}{2}|\nabla u|^2 m(t) + F(m(t)) - Vm(t)  \, dx \ = E.
		\]
		Note that if $\nabla m_0 \in L^1(\R^N)$, then $\nabla u(0) \nabla m_0 \in L^1(\R^N)$, and therefore the previous equality holds also for $t = 0$.
		
		\medskip

		Now we prove claim {\it (ii)}. By  the Young's inequality, we have
		\begin{align*}
		\int_{\R^N}\nabla u(0) \cdot \nabla m_0 \, dx+&\frac{1}{2}\int_{\R^N}|\nabla u(0)|^2m_0 \, dx \\ \geq& -\frac{1}{2}\int_{\R^N}\frac{|\nabla m_0|^2}{m_0}\, dx-\frac{1}{2}\int_{\R^N}|\nabla u(0)|^2m_0\, dx+\frac{1}{2}\int_{\R^N}|\nabla u(0)|^2m_0 \, dx\\=&-\frac{1}{2}\int_{\R^N}\frac{|\nabla m_0|^2}{m_0}\, dx.
		\end{align*}
		Then by claim {\it i)}, we obtain
		$$
		E\geq -\frac{1}{2}\int_{\R^N}\frac{|\nabla m_0|^2}{m_0}\, dx+\int_{\R^N}F(m_0)\, dx-\int_{\R^N} Vm_0 \, dx,
		$$
		which is equivalent to {\it ii)}.

\end{proof}

	The next lemma is crucial. Exploiting the structure of the MFG system, it is possible to evaluate the second derivative in time of second order moments of $m$ in terms of integral quantities related to $f, V$ and $m$.
		
		 \begin{lem}\label{lem:conserv} Let $(u, m)$ be a classical solution of \eqref{2}. Under the assumptions of Lemmas \ref{lem:intmnablam}, \ref{lem:ddtmoment} and \ref{lem:E}, 
		 $t \mapsto \int_{\R^N} x^2 m(t) \, dx$ is of class $C^2$, and for all $t$,
		 	\begin{multline*}
			\frac{d^2}{dt^2}\left(\int_{\R^N} x^2m(t)\, dx\right)=4E+2N\int_{\R^N}f(m)m\, dx-2(N+2)\int_{\R^N}F(m)\, dx\\+4 \int_{\R^N}Vm\, dx+2\int_{\R^N}\nabla V \cdot x m \, dx.
			\end{multline*}
	\end{lem}

\begin{proof}

   We multiply the second equation in \eqref{2} by $x \cdot \nabla u$ and the first one by $x\cdot \nabla m$, and integrate over $B_R \times (t_1, t_2)$ to get
   $$
   \int_{t_1}^{t_2}\int_{B_R}\left[m_t-\Delta m -\mbox{div}(\nabla um)\right] x\cdot \nabla u+ \left[ -u_t-\Delta u +\frac{1}{2}|\nabla u|^2+f(m)-V\right] x \cdot \nabla m\, dx dt=0,
   $$
   which is equivalent to
   \begin{align}\label{eq:claim31}
   \int_{t_1}^{t_2} \int_{B_R} m_t x \cdot \nabla u-u_t x\cdot \nabla m\,dxdt & = \int_{t_1}^{t_2} \int_{B_R} \Delta m x \cdot \nabla u\, dx dt +  \int_{t_1}^{t_2} \int_{B_R}\Delta u x \cdot \nabla m\, dx dt \\ 
   + & \int_{t_1}^{t_2}\int_{B_R}\di(\nabla u m) x \cdot \nabla u\, dx dt -\frac{1}{2}  \int_{t_1}^{t_2} \int_{B_R} |\nabla u|^2 x \cdot \nabla m\, dx dt \nonumber \\ - &  \int_{t_1}^{t_2} \int_{B_R} f(m) x \cdot \nabla m \, dx dt+\int_{B_R}V x \cdot \nabla m \, dx dt  \nonumber.
   \end{align}
   We start with the first two terms of the right-hand side. Integrating by parts, one has
   \begin{align*}
   \int_{B_R} \Delta m x \cdot \nabla u\, dx+ \int_{B_R} \Delta u x \cdot \nabla m\, dx =& -\int_{B_R} \nabla m \cdot \nabla(x \cdot \nabla u)\,dx-\int_{B_R} \nabla u \cdot \nabla(x \cdot \nabla m)\, dx \\ & + \int_{\partial B_R}\left[\nabla m  (x \cdot \nabla u) +\nabla u(x \cdot \nabla m)\right]\cdot n\, dx.
   \end{align*}
   Note that $
   \int_{B_R} x \cdot \nabla(\nabla m \cdot \nabla u) =\int_{\partial B_R} x (\nabla m \cdot \nabla u)\cdot n-\int_{ B_R} \nabla m \cdot \nabla u \mbox{ div}x $,
   so
   $$
   \begin{aligned}
   \int_{B_R} \nabla m \cdot \nabla(x \cdot \nabla u) +\nabla u \cdot \nabla(x \cdot \nabla m)\, dx&= \int_{B_R} m_{x_i} \partial_{x_i}(x_k u_{x_k})+u_{x_i} \partial_{x_i}(x_k m_{x_k}) \, dx\\
   &= \int_{B_R} 2m_{x_i} u_{x_i} +m_{x_i} x_k u_{x_k x_i} +u_{x_i} x_k m_{x_kx_i} \, dx\\
   &= 2\int_{B_R} \nabla m \cdot \nabla u  \, dx+\int_{B_R} x \cdot \nabla(\nabla m \cdot \nabla u)  \, dx\\
   &=\int_{B_R} \nabla m \cdot \nabla u (2-N) \, dx+\int_{\partial B_R} x (\nabla m \cdot \nabla u)\cdot n \, dx,
   \end{aligned}
   $$
   and therefore
   \begin{multline}\label{eq:lemotto0}
   \int_{t_1}^{t_2}  \int_{B_R} \Delta m x \cdot \nabla u\, dx dt+ \int_{t_1}^{t_2} \int_{B_R} \Delta u x \cdot \nabla m\, dx dt = \int_{t_1}^{t_2} \int_{B_R} \nabla m \cdot \nabla u (N-2) \, dx dt+  \\ \int_{t_1}^{t_2} \int_{\partial B_R} \Big[\nabla m  (x \cdot \nabla u) + \nabla u(x \cdot \nabla m) -x (\nabla m \cdot \nabla u)\Big]\cdot n \, dx dt.
   \end{multline}
   
   To handle the third and fourth term of the right-hand side of \eqref{eq:claim31}, the following formula will be useful
   $$
   \frac{1}{2}\nabla (|\nabla u|^2)\cdot x= \nabla u \cdot \nabla(\nabla u \cdot x)-|\nabla u|^2.
   $$
   Then, integrating by parts
   \begin{align*}
   -\frac{1}{2}\int_{B_R}|\nabla u|^2x \cdot \nabla m \, dx & = -\frac{1}{2}\int_{\partial B_R} |\nabla u|^2 m x  \cdot n\, dx + \frac{1}{2}\int_{B_R}\di\big(|\nabla u|^2 x\big)  m \, dx \\ 
   & = -\frac{1}{2}\int_{\partial B_R} |\nabla u|^2 m x  \cdot n\, dx +\frac{N}{2}\int_{B_R}|\nabla u |^2m \, dx\\ 
   & \qquad \qquad +\int_{B_R}\nabla u \cdot \nabla(\nabla u \cdot x)m\, dx-\int_{B_R}|\nabla u|^2m\, dx.
   \end{align*}
   Since
   $$
   \int_{B_R}\nabla u \cdot \nabla(\nabla u \cdot x)m\, dx=\int_{\partial B_R} \nabla u(\nabla u \cdot x) m \cdot n \, dx-\int_{B_R}\di(\nabla u m) \nabla u \cdot x\, dx,
   $$
   we obtain
   \begin{multline}\label{lem:formula2}
   \int_{t_1}^{t_2} \int_{B_R}\di(\nabla u m)\nabla u \cdot x \, dx dt -\frac{1}{2}\int_{t_1}^{t_2}\int_{B_R}|\nabla u|^2x \cdot \nabla m \, dxdt = \left(\frac{N}{2}-1\right)\int_{t_1}^{t_2} \int_{B_R}|\nabla u |^2m \, dx dt \\
   +\int_{t_1}^{t_2} \int_{\partial B_R}   \left [\nabla u(\nabla u \cdot x) m  -\frac{1}{2} |\nabla u|^2 m x \right] \cdot n\, dx dt.
   \end{multline}
   
   Regarding the last two terms in the right hand side of \eqref{eq:claim31}, by the definition of $F$ and by integrating by parts we have 
   \begin{multline}\label{part1}
   - \int_{t_1}^{t_2}\int_{B_R}f(m)\nabla m \cdot x \, dxdt + \int_{t_1}^{t_2} \int_{B_R} V x \cdot \nabla m \, dx dt = \\
   N\int_{t_1}^{t_2}\int_{B_R}F(m) \, dxdt -\int_{t_1}^{t_2} \int_{B_R} \nabla V \cdot x m\, dx dt-N\int_{t_1}^{t_2} \int_{B_R} Vm\, dx dt + \\
   \int_{t_1}^{t_2} \int_{\partial B_R} \big[- F(m) + V m \big] x  \cdot n \, dx dt .
   \end{multline}
   
   We now manipulate the left-hand side of \eqref{eq:claim31}. We perform a first integration by parts to have
   $$
   -\int_{B_R} u_t x\cdot \nabla m\, dx=\int_{B_R}  \nabla u_t\cdot xm\, dx + N\int_{B_R}u_tm\, dx - \int_{\partial B_R} u_t mx \cdot n\, dx.
   $$
   Then,
   \begin{multline*}
   \int_{t_1}^{t_2} \int_{B_R} m_t x\cdot \nabla u-u_t x\cdot \nabla m\,dx = \int_{t_1}^{t_2} \int_{B_R} \partial_t \big(m \nabla u \big)\cdot x\, dx dt + N\int_{t_1}^{t_2} \int_{B_R} u_tm\, dx dt \\ - \int_{t_1}^{t_2} \int_{\partial B_R} u_t  mx \cdot n\, dx dt.
   \end{multline*}
   The equation for $u$ yields
   \begin{multline*}
   N\int_{t_1}^{t_2}\int_{B_R} u_tm\, dx dt=N\int_{t_1}^{t_2}\int_{B_R}\left [-\Delta u+\frac{1}{2}|\nabla u |^2+f(m)-V\right] m\, dxdt \\ = N\int_{t_1}^{t_2}\int_{B_R}\nabla u \cdot \nabla m+\frac{1}{2}|\nabla u |^2m+f(m)m-V m\, dxdt - 
   N \int_{t_1}^{t_2}\int_{\partial B_R} \nabla u m \cdot n \, dxdt,
   \end{multline*}
   and therefore
   \begin{multline}\label{part2}
   \int_{t_1}^{t_2} \int_{B_R} m_t x\cdot \nabla u-u_t x\cdot \nabla m\,dx dt= \int_{B_R}[ m(t_2)\nabla u(t_2) - m(t_1) \nabla u(t_1)]\cdot x \,dx + \\ N\int_{t_1}^{t_2}\int_{B_R}\nabla u \cdot \nabla m+\frac{1}{2}|\nabla u |^2m+f(m)m-V m\, dxdt - 
   \int_{t_1}^{t_2}\int_{\partial B_R}\left[ N \nabla u m  + u_t m x \right]\cdot n\, dx dt.
   \end{multline}
   By plugging  \eqref{eq:lemotto0}, \eqref{lem:formula2}, \eqref{part1} and \eqref{part2} into \eqref{eq:claim31}, we obtain
   \begin{multline}
   \int_{B_R}[ m(t_2)\nabla u(t_2) - m(t_1) \nabla u(t_1)]\cdot x \,dx = -2  \int_{t_1}^{t_2} \int_{B_R} \nabla m \cdot \nabla u \, dx dt- \int_{t_1}^{t_2}\int_{B_R}  |\nabla u |^2m\, dxdt \\
   + N \int_{t_1}^{t_2}\int_{B_R}F(m)- f(m)m \, dxdt -\int_{t_1}^{t_2} \int_{B_R} \nabla V \cdot x m\, dx dt + G_R,
   \end{multline}
   where
   \begin{align*} G_R = \int_{t_1}^{t_2} \int_{\partial B_R} & \big[-  F(m)+ V m \big] x  \cdot n +
   \big[ N \nabla u m + u_t m x\big] \cdot n + \\
   & \big[\nabla m  (x \cdot \nabla u) + \nabla u(x \cdot \nabla m) -x (\nabla m \cdot \nabla u)\big]\cdot n + \\
   & + \left[\nabla u(\nabla u \cdot x) m  -\frac{1}{2} |\nabla u|^2 m x \right] \cdot n\, dx dt.
   \end{align*}
   Since $V, \nabla u, u_t\in L^\infty(\R^N\times (0,T))$ and $m, |\nabla u| \, |\nabla m| \in L^1(\R^N \times (0,T))$, by Lemma \ref{lem:lem32time} we have
   $$
   \lim_{R \to \infty} G_R=0.
   $$
   Then, since $m, \nabla u \cdot \nabla m \in L^1(\R^N\times (0,T))$ and $\nabla u, m, \nabla V \in L^\infty(\R^N\times (0,T))$, we obtain
   \begin{multline}\label{eqgen}
   \int_{\R^N}[ m(t_2)\nabla u(t_2) - m(t_1) \nabla u(t_1)]\cdot x \,dx = -2 \int_{t_1}^{t_2} \int_{\R^N} \nabla m \cdot \nabla u +  \frac 12 |\nabla u |^2m\, dxdt  \\
   + N \int_{t_1}^{t_2}\int_{\R^N}F(m)- f(m)m \, dxdt -\int_{t_1}^{t_2} \int_{\R^N} \nabla V \cdot x m\, dx dt.
   \end{multline}
We now use Lemma \ref{lem:ddtmoment} (i) to rewrite the left-hand side, and Lemma \ref{lem:E} (i) to replace the first two terms of the right hand side to obtain for a.e. $t_1, t_2$
\begin{multline}
         \frac{d}{dt}\frac12\left(\int_{\R^N} x^2m(t_1)\, dx\right)-\frac{d}{dt}\frac12 \left(\int_{\R^N} x^2m(t_2)\, dx\right) = -2E(t_2-t_1)-2 \int_{t_1}^{t_2} \int_{\R^N} V m \, dxdt  \\
 + \int_{t_1}^{t_2}\int_{\R^N}(N+2)F(m)- N f(m)m\, dxdt -\int_{t_1}^{t_2} \int_{\R^N} \nabla V \cdot x m\, dx dt.
\end{multline}
Then, dividing by $t_2-t_1$ and taking the limit $t_2 \to t_1$ we obtain the desired equality for a.e. $t$. Note that $m \in C(L^1)$ and $m$ is bounded, so $\frac{d}{dt}\int_{\R^N}x^2m(t)\, dx$ agrees a.e. with a $C^1$ function in $t$. Hence, $t \mapsto \int_{\R^N} x^2m(t)\, dx$ is of class $C^2$.

\end{proof}

	We are now ready to prove Theorem \ref{thm:noexist}.
\begin{proof}[Proof of Theorem \ref{thm:noexist}] Define $h(t):=\int_{\R^N}  x^2m(t) \, dx$, which is of class $C^2$ in view of Lemma \ref{lem:conserv}. We first claim that $$h(0)=h_0 > 0, \qquad h'(T) \le 2N, \qquad h''(t)\geq 4 e_0>0 \, \,  \forall t\geq 0.$$
Indeed, $h(0) = \int_{\R^N} x^2 m_0  \, dx =: h_0 > 0$, by Lemma \ref{lem:ddtmoment} and $\nabla u_T \cdot x \geq 0$, we obtain
	$
	h'(T) \leq2N.
	$
	Note that by Lemma \ref{lem:E} {\it ii)} and the assumptions \eqref{Vass}, \eqref{fass}, \eqref{e0ass} (and the fact that $\int m(t) dx = 1$ for all $t$) we have
	\begin{multline}\label{condfuse}
	4E+2N\int_{\R^N}f(m)m\, dx-2(N+2)\int_{\R^N}F(m)\, dx+4 \int_{\R^N}Vm\, dx+2\int_{\R^N}\nabla V \cdot x m \, dx \ge \\
	4\left[-\frac{1}{2}\int_{\R^N} \frac{|\nabla m_0|^2}{m_0}\, dx+\int_{\R^N} F(m_0)\, dx-\int_{\R^N} (V- \inf_{\R^N}V)m_0\, dx \right] + \\ 2\int_{\R^N}N f(m)m\, dx-(N+2)F(m)\, dx+2 \int_{\R^N}[2(V- \inf_{\R^N}V) + \nabla V \cdot x]m\, dx \\
	\geq 4 e_0>0.
	\end{multline}
	Then, by Lemma \ref{lem:conserv} and \eqref{condfuse}, we obtain
	$
	h''(t) \geq 4 e_0>0
	$
	for all $t$.

   Now we define $v(t)=2e_0 t^2+(2N-4e_0T)t+h_0$ and observe that
	$$
	v(0)=h_0, \qquad v'(T)=2N, \qquad v''(t)=4e_0.
	$$
	Then, by comparison, we derive
	\begin{equation}\label{comparison}
	h(t) \leq v(t) \quad \forall t \in [0,T].
	\end{equation}
	Let now $\bar t = T-\frac{N}{2e_0}$. We have $\bar t \in [0,T]$ if $T \ge \frac{N}{2e_0}$. Moreover,
	$$
	v(\bar t) = -2e_0 T^2+2NT-\frac{N^2}{2e_0} +h_0 = -\left(\sqrt{2e_0} T -  \frac{N}{\sqrt{2e_0}}\right)^2 + h_0< 0 
	$$
	provided that $T > \frac{N}{2e_0} +\sqrt{ \frac{h_0}{2e_0}} =: T_*$. Therefore, if $T > T_*$,
	$$
	h(\bar t) \leq v(\bar t)<0,
	$$
	that leads to a contradiction since $0\leq h(t)=\int_{\R^N} x^2 m(t) \, dx$  for all $t \in [0,T]$.
\end{proof}

\begin{rem}[Minimizing the non-existence time horizon]\label{Timprove} \upshape We observe that the lower bound on the time horizon $T$ yielding non-existence
$$ \frac{N}{2e_0} + \sqrt{\frac{\int_{\R^N}m_0 x^2 dx}{2e_0}}$$
can be ``optimized''. We exploit in particular the obvious fact that non-existence to \eqref{2} holds if and only if non-existence holds for the translated system
\begin{equation}\label{systrans}
\begin{cases}
-u_t-\Delta u+ \frac{1}{2} |\nabla u|^2= -f(m) +V_y\\
m_t -\Delta m -\di(\nabla um)=0 \\
m(0)=m_{0,y}, \quad u(T)=u_{T,y}
\end{cases} \qquad \text{for all $y \in \R^N$,}
\end{equation}
where
$$
V_y(x) = V(x+y), \quad m_{0,y}(x) = m_0(x+y), \quad u_{T,y}(x) =  u_T(x+y).
$$
Assume that there exists $\Omega \subset \R^N$ such that
$$
2(V(x+y)- \inf_{\R^N}V) + \nabla V(x+y) \cdot x \geq 0, \quad \nabla u_T(x+y) \cdot x \geq 0 \quad \text{for all $y \in \Omega$}.
$$
Then, if $m_0, u_T, V$ satisfy \eqref{fass}, \eqref{Vass}, \eqref{dataass}, \eqref{dataass2}, \eqref{e0ass}, then $m_{0,y}, u_{T,y}, V_{y}$ also satisfy \eqref{fass}, \eqref{Vass}, \eqref{dataass}, \eqref{dataass2}, \eqref{e0ass} for all $y \in \Omega$ (note in particular that \eqref{e0ass} is translation invariant). Since Theorem \ref{thm:noexist} guarantees non-existence of solutions to \eqref{systrans} for 
$$
T > \frac{N}{2e_0} + \sqrt{\frac{\int_{\R^N}m_{0,y} x^2 dx}{2e_0}}\quad \text{for all $y \in \Omega$},
$$
we can conclude that the original MFG system \eqref{2} has no classical solutions whenever
$$
T > \frac{N}{2e_0} + \sqrt{\frac{\displaystyle\inf_{y \in \Omega}\int_{\R^N} (x-y)^2 m_{0}(x) dx}{2e_0}}.
$$

As a simple illustration, consider $u_T \equiv V \equiv 0$. It is clear that the quantity $y \mapsto \int_{\R^N} (x-y)^2 m_{0}(x) dx$ might be minimized by $y = 0$, but this may be not the case for non-radially symmetric $m_0$.

\end{rem}

\begin{rem}[The non-quadratic case]\label{nonquad} \upshape Consider a more general MFG system with power-like Hamiltonian $H(p) = \frac1\gamma|p|^\gamma$, $\gamma > 1$, i.e.
$$
\begin{cases}
-u_t-\Delta u+ \frac{1}{\gamma} |\nabla u|^\gamma=-f(m) +V(x) & \mbox{ in } \R^N \times (0,T),\\
m_t -\Delta m -\di(|\nabla u|^{\gamma-2}\nabla um)=0 & \mbox{ in }  \R^N \times (0,T), \\
m(0)=m_0, \quad u(T)=u_T & \mbox{ on } \R^N.
\end{cases}
$$
We observe that though the procedures described above yield meaningful identities also when $\gamma \neq 2$, it is not clear how to conclude similar non-existence results. First, since the Hamiltonian nature of the MFG system is independent of $\gamma > 1$, one still has a conservation of energy of the form
\begin{equation}\label{eqE}
\int_{\R^N} \nabla u \cdot \nabla m \, dx+\frac{1}{\gamma}\int_{\R^N} |\nabla u|^\gamma m\, dx+\int_{\R^N} F(m) \, dx-\int_{\R^N}Vm \, dx = E
\end{equation}
for all $t$. Moreover, arguing as in Lemma \ref{lem:ddtmoment},
\begin{equation}\label{eqddt}
\frac {d^2}{dt^2}\int_{\R^N} m(t) x^2\, dx =  -  \frac d{dt}  \int_{\R^N} m(t) |\nabla u|^{\gamma-2} \nabla u \cdot x\, dx.
\end{equation}
Finally, testing the equations by  $x\cdot \nabla m$ and $x \cdot \nabla u$ respectively, and reasoning as in Lemma \ref{lem:conserv}, one obtains
\begin{multline}\label{eqzzz}
        \frac d{dt} \int_{\R^N}m\nabla u \cdot x \,dx = -2  \int_{\R^N} \nabla m \cdot \nabla u +  \frac 1\gamma |\nabla u |^\gamma m\, dx  \\
 + N \int_{\R^N}F(m)- f(m)m \, dx - \int_{\R^N} \nabla V \cdot x m\, dx + \left(  \frac 2\gamma -1  \right)\int_{\R^N} |\nabla u |^\gamma m \, dx,
\end{multline}
which is basically \eqref{eqgen} with an additional term $ \left(  \frac 2\gamma -1  \right)\int |\nabla u |^\gamma m$ (that has a sign!). While for any $\gamma$ it is possible to plug the energy identity \eqref{eqE} into \eqref{eqzzz}, it is not clear how to couple \eqref{eqddt} with \eqref{eqzzz} to get an identity for $\frac {d^2}{dt^2}\int m(t) x^2$ (or other similar quantitities) that give information on its sign when $\gamma \neq 2$ (and therefore strict convexity of $t \mapsto \int m(t) x^2$, which is the key point in the non-existence argument).

\end{rem}

\subsection{The Planning Problem}\label{splanning}

The Planning Problem in MFG gives rise to the following system
\begin{equation} \label{planning}
\begin{cases}
-u_t-\Delta u+ \frac{1}{2} |\nabla u|^2=-f(m)+ V(x) \quad \quad  & \mbox{ in } \R^N \times (0,T),\\
m_t -\Delta m  -\di(\nabla u m)=0 &  \mbox{ in } \R^N \times (0,T),\\
m(0)=m_0, \quad m(T)=m_T & \mbox{ in } \R^N,
\end{cases}
\end{equation}
where $\int_{\R^N} m_0\, dx=\int_{\R^N} m_T\, dx=1, m_0, m_T\ge0$. Roughly speaking, in a typical planning problem, one wants to drive the density of players from an initial configuration $m_0$ to a target final  one $m_T$. In  \cite{PLLions}, existence and uniqueness of smooth solutions to \eqref{2} is discussed. In \cite{PLPorr}, it is proven the existence of weak solutions when the Hamiltonian is not necessarily quadratic in $\nabla u$. Note that both references consider the monotone case ($-f(m)$ increasing) only.

We apply similar arguments as in the proof of Theorem \ref{thm:noexist} to prove non  existence for large time for the problem \eqref{planning}.
\begin{thm}\label{thm:noexistpl}
Assume \eqref{Vass}, and that \eqref{dataass2} holds both for $m_0$ and $m_T$. Then, if 
$$
T>\sqrt{\frac{2\max\Big\{\int_{\R^N}x^2m_0  dx, \,\int_{\R^N}x^2m_T  dx\Big\}}{e_0}},
$$
the system \eqref{planning} has no classical solutions. 
\end{thm}

\begin{proof} We follow the proof of Theorem \ref{thm:noexist}, assuming by contradiction that a solution exists. The conclusions of Lemmas \ref{lem:intmnablam}, \ref{lem:ddtmoment}, \ref{lem:E} and Lemma \ref{lem:conserv} hold true, since the information $u(T) = u_T$ is never invoked throughout their proofs. Therefore, as in \eqref{condfuse}, we have
		$$h''(t) = \frac{d^2}{dt^2} \int_{\R^N}x^2m(x,t) \, dx \geq 4 e_0 > 0$$ for all $t\geq 0$.
	The main difference is  that we do not have any information on $u(T)$, hence we cannot infer any information on the sign of $h'(T)$. On the contrary now 
	\begin{equation}\label{hT}
	h_T:=h(T) = \int_{\R^N}x^2m_T(x)\, dx\geq 0,
	\end{equation}
	is a datum of the problem.
	Therefore, we define $v(t):=2e_0t(t-T)+\max\{h_0, h_T\}$ and remark that
	$$
	v(0)\geq h_0, \quad v(T)\geq h_T, \quad v''(t)=4e_0,
	$$
	hence by comparison we deduce
	$$
	h(t) \leq v(t) \quad \forall t \in [0,T].
	$$
	Noting that 
	$$
	h\Big(\frac T2\Big) \le v\Big(\frac T2\Big)=-\frac{e_0}{2}T^2+\max\{h_0, h_T\}<0
	$$
	whenever
	$T>\widehat T := \sqrt{\frac{2\max\{h_0,h_T\}}{e_0}} $, we obtain the desired contradiction, since $h$ has to be nonnegative. Note that one could choose $v$ in a way that it satisfies $v''(t)=4e_0$ and agrees with $h_0$ and $h_T$ at time $t=0$ and $t=T$ respectively, to improve $\widehat T$ when $h_0 \neq h_T$ (but we avoid writing the computations here for the sake of simplicity). \end{proof}

\section{Existence}
In this section we prove the existence Theorem \ref{thm:exvmu}. Note first that \eqref{eq:ipofex} implies
$$
0\leq f(m)\leq \sigma m^\alpha, \quad 0\leq F(m)\leq \frac{\sigma}{\alpha+1}m^{\alpha+1}.
$$
We will use a generalization of the Schauder fixed point theorem, that we recall here for completeness (see Theorem 5.1 of \cite{FxpBook}).
\begin{thm} \label{prop:fixpoint}
	Let $X$ be a Banach space, $C\subset X$  closed and convex.
	Let $U$ be an open subset of $C$ and $p \in U$. Consider a map $\mathcal{F} : \overline U \mapsto C$ continuous and compact.  Suppose that 
	$$
	u=\eta \mathcal{F}(u)+(1-\eta)p, \mbox{ for some } \eta \in (0,1)  \Rightarrow u \notin \partial U.
	$$
	Then $\mathcal{F}$ has a fixed point in $\overline U$.
\end{thm}

In order to prove Theorem \ref{thm:exvmu}, we will need the following crucial a priori estimate.
\begin{thm}\label{thm:estex}
Under the assumptions of Theorem \ref{thm:exvmu}, let $$(\eta, v,\mu) \in (0,1) \times C([0,T]; C_b^4(\R^d)) \times C([0,T]; C_b^1 \cap L^1(\R^d))$$
be a smooth solution of 
\begin{equation}\label{3}
\begin{cases}
-u_t-\Delta v+ \frac{1}{2} |\nabla u|^2=-f(\mu)+V(x)  & \mbox{ in } \R^N \times (0,T),\\
\mu_t -\Delta \mu -\di(\nabla u \mu)=0 &  \mbox{ in } \R^N \times (0,T),\\
\mu(0)=\eta m_0, \quad u(T)=u_T & \mbox{ in } \R^N.
\end{cases}
\end{equation}
	 Then there exists $ \delta \in \left(1,2\right)$ depending on $\alpha$ and $C>0$ depending on $N, \alpha, \|m_0\|_{L^{\alpha+1}(\R^N)}, \|u_T\|_{C^2(\R^N)}$ such that
	$$
	D^{\delta}\leq C\left(\sigma^2 D^2 +T||\Delta V||_{L^\infty(\R^N)}D+1\right),
	$$
	where
	$$
	D=\int_0^T\int_{\R^N}\mu^{2\alpha+1}\, dxdt.
	$$
	\end{thm}

\begin{proof} {\bf Step 1: regularity of $\mu$ given by the equation}. We multiply the second equation of \eqref{3} by $\mu^\alpha$ and integrate by parts to obtain
\begin{align*}
\frac{1}{\alpha+1}\int_{B_R}\mu^{\alpha+1} (t)\, dx+\alpha\int_0^t\int_{B_R}\mu^{\alpha-1}|\nabla \mu|^2\, dxdt\nonumber = & -\alpha\int_0^t\int_{B_R}\nabla v \cdot\nabla \mu \mu^\alpha\, dxdt\\ & + \frac{1}{\alpha+1}\int_{B_R}\mu^{\alpha+1}(0)\, dx+G_R, 
\end{align*}
where
$$
G_R=\int_0^t\int_{\partial B_R}\nabla \mu \mu^\alpha \cdot n\, dxdt+\int_0^t\int_{\partial B_R} \nabla v \mu \mu^\alpha \cdot n \, dx dt.
$$
Since $\mu, \nabla \mu, \nabla v$ are bounded, and $\mu \in C(L^1)$, we can apply Lemma \ref{lem:lem32time} and the monotone convergence theorem to pass to the limit $R_n \to \infty$, and to obtain for any $t \in [0,T]$
\begin{align}\label{firsteq}
\frac{1}{\alpha+1}\int_{\R^N}\mu^{\alpha+1}(t)\, dx+\alpha\int_0^t\int_{\R^N}\mu^{\alpha-1}|\nabla \mu|^2\, dxdt\nonumber =&-\alpha\int_0^t\int_{\R^N}\nabla v \cdot\nabla \mu \mu^\alpha\, dxdt\\ & + \frac{1}{\alpha+1}\int_{\R^N}\mu^{\alpha+1}(0)\, dx,
\end{align}
showing in particular that $\mu^{\alpha-1}|\nabla \mu|^2 \in L^1(\R^N \times (0,T))$. Writing $\nabla \mu \mu^\alpha=\frac{1}{\alpha+1}\nabla \mu^{\alpha+1}$, we can integrate by parts again (Lemma \ref{lem:intp}, $\mu, \nabla \mu, \nabla v, D^2 v$ are bounded and $\mu \in C(L^1)$) and use H\"older inequality to get
\begin{multline}\label{secintparts}
-\alpha\int_0^t\int_{\R^N}\nabla v \cdot \nabla \mu \mu^\alpha\, dxdt=
\frac{\alpha}{\alpha+1}\int_0^t\int_{\R^N}\Delta v\mu^{\alpha+1}\, dxdt \\ \le \frac{\alpha}{\alpha+1} \left(
\int_0^t\int_{B_R}|D^2v|^2\mu \, dxdt\right)^{\frac{1}{2}}\left(\int_0^t\int_{B_R}\mu^{2\alpha+1}\, dx dt\right)^{\frac{1}{2}}.
\end{multline}
By writing $\mu^{\alpha-1}|\nabla \mu|^2=\frac{4}{(\alpha+1)^2}|\nabla \mu^{\frac{\alpha+1}{2}}|^2$, and putting \eqref{secintparts} into \eqref{firsteq}, we obtain (letting $t$ vary in $[0, T]$)
\begin{multline}\label{step1}
\left[ \sup_{t \in [0,T]} \int_{\R^N}\mu^{\alpha+1}\, dx + \frac{4\alpha}{\alpha+1}\int_0^T\int_{\R^N}|\nabla \mu^{\frac{\alpha+1}{2}}|^2\, dxdt \right]^2 \leq \\ 4\int_{\R^N}\mu^{\alpha+1}(0)\, dx + 4\alpha^2\left(\int_0^T\int_{\R^N}|D^2v|^2\mu \, dxdt\right) \left(\int_0^T\int_{\R^N}\mu^{2\alpha+1}\, dx dt\right).
\end{multline}

{\bf Step 2: parabolic interpolation.} By Proposition 3.1, inequality (3.2) of \cite{DiB93} with $v=\mu^{\frac{\alpha+1}{2}}, p=2, q=\frac{2(2\alpha+1)}{\alpha+1}$ and $\frac{N+m}{N}=\frac{2\alpha+1}{\alpha+1}$, there exists some $\tilde C>0$ depending on $N$ and $\alpha$ (but not on $T$), such that
$$
\left(\int_0^T\int_{\R^N}\mu^{2\alpha+1}\, dx dt\right)^{\frac2q}\leq \tilde C\left[\int_0^T\int_{\R^N}|\nabla \mu^{\frac{\alpha+1}{2}}|^2\, dxdt+\left(\sup_{t \in [0,T]} \int_{\R^N}\mu^{\beta}\, dx\right)^{\frac{2}{m}}\right],
$$
where $\beta=m \frac{\alpha+1}{2}=\frac{\alpha N}{2}$ and 
$$
\frac{2}{q}=\frac{\alpha+1}{2\alpha+1} \in \left(\frac{1}{2},1\right).
$$
Moreover, note that if $N\leq 2$ and $\alpha\geq \frac{2}{N}$ or if $N >2$ and $\frac{2}{N}\leq\alpha \leq\frac{2}{N-2}$, then $1\leq\beta\leq\alpha+1$. By the interpolation inequality
$$
||\mu(t)||_{L^\beta(\R^N)}\leq ||\mu(t)||_{L^1(\R^N)}^{1-\theta}||\mu( t)||_{L^{\alpha+1}(\R^N)}^{\theta}, \quad \frac{1}{\beta}=1-\theta +\frac{\theta}{\alpha+1},
$$
and by recalling that $||\mu(t)||_{L^1(\R^N)}\leq1$ for all $t$ (since $||\mu(t)||_{L^1(\R^N)}=||\mu(0)||_{L^1(\R^N)}\leq||m_0||_{L^1(\R^N)}$ ), we get
$$
\left(\int_0^T\int_{\R^N}\mu^{2\alpha+1}\, dx dt\right)^{\frac{2}{q}}\leq \tilde C\left[\int_0^T\int_{\R^N}|\nabla \mu^{\frac{\alpha+1}{2}}|^2\, dxdt+\left(\sup_{t \in [0,T]} \int_{\R^N}\mu^{\alpha+1}\, dx\right)^{a}\right],
$$
where $a=\frac{2\theta \beta}{m (\alpha+1)}$. Since $a< 1$, we have
\begin{equation}\label{m2alpha+1}
\left(\int_0^T\int_{\R^N}\mu^{2\alpha+1}\, dx dt\right)^{\frac{4}{q}}\leq 2\tilde C^2\left[\int_0^T\int_{\R^N}|\nabla \mu^{\frac{\alpha+1}{2}}|^2\, dxdt+ \sup_{t \in [0,T]} \int_{\R^N}\mu^{\alpha+1}\, dx \right]^2 + 2\tilde C^2.
\end{equation}
Therefore, plugging this inequality into \eqref{step1} gives
\begin{multline} \label{B}
\left(\int_0^T\int_{\R^N}\mu^{2\alpha+1}\, dx dt\right)^{\frac{4}{q}} + \left[\int_0^T\int_{\R^N}|\nabla \mu^{\frac{\alpha+1}{2}}|^2\, dxdt+ \sup_{t \in [0,T]} \int_{\R^N}\mu^{\alpha+1}\, dx\right]^2 \leq \\
c\left[ \int_{\R^N}\mu^{\alpha+1}(0)\, dx + \left(\int_0^T\int_{\R^N}|D^2v|^2\mu \, dxdt\right) \left(\int_0^T\int_{\R^N}\mu^{2\alpha+1}\, dx dt\right) + 1 \right]
\end{multline}
for some $c$ depending on $N, \alpha$.

 {\bf Step 3: ``second order'' estimates for the MFG system}. Computing the Laplacian of the equation for $v$ yields
$$
-\partial_t \Delta v-\Delta\Delta v+|D^2v|^2+\nabla(\Delta v)\cdot \nabla v=-\mbox{div}(f'(\mu)\nabla \mu)+ \Delta V.
$$
Recall that $\mu, \nabla \mu$, and space derivatives of $v$ are bounded up to the fourth order, and $\mu \in C(L^1)$. Hence, we multiply by $\mu$ and integrate by parts using Lemma \ref{lem:intp} to obtain
\begin{multline*}
-\int_0^T\int_{\R^N}\partial_t \Delta v \mu \, dx dt-\int_0^T\int_{\R^N}\Delta v \big(\Delta \mu + \di(\nabla v \mu) \big) \, dx dt + \int_0^T \int_{\R^N}|D^2v|^2\mu \, dx dt \\ = \int_0^T\int_{\R^N}f'(\mu)|\nabla \mu|^2\, dx dt+\int_0^T\int_{\R^N} \Delta V\mu\, dxdt.
\end{multline*}
From \eqref{eq:ipofex} it follows
$$
 \int_0^T\int_{\R^N}f'(\mu)|\nabla \mu|^2\, dx dt\leq \sigma \alpha \int_0^T\int_{\R^N}\mu^{\alpha-1}|\nabla \mu|^2\, dx dt
$$ 
Using  the equation for $\mu$ in the left-hand term
\begin{equation*}
-\int_0^T\int_{\R^N}\partial_t \Delta v \mu \, dx dt-\int_0^T\int_{\R^N}\Delta v \big(\Delta \mu + \di(\nabla v \mu) \big) \, dx dt =\int_{\R^N}\Delta v(0)\mu(0)\, dx - \int_{\R^N}\Delta v(T)\mu(T)\, dx
\end{equation*}
and plugging the previous inequalities  implies
\begin{multline}
\int_0^T\int_{\R^N}|D^2v|^2\mu \, dx dt\leq \frac{4\sigma\alpha}{(\alpha+1)^2} \int_0^T\int_{\R^N}|\nabla \mu^{\frac{\alpha+1}{2}}|^2\, dx dt+\int_{\R^N}\Delta v(T)\mu(T)\, dx+\int_{\R^N}\nabla v(0)\nabla \mu(0)\, dx\\+\int_0^T\int_{\R^N}\Delta V \mu \, dx dt. \label{A}
\end{multline}
We now control the term $\int\nabla v(0)\nabla \mu(0)\, dx$. By Lemma \ref{lem:E} (i) and integration by parts we have
\begin{align*}
& \int_{\R^N}\nabla v(0)\nabla \mu(0)\, dx   + \frac{1}{2}\int_{\R^N}|\nabla v(0)|^2\mu(0)\, dx+\int_{\R^N}F(\mu(0))\, dx-\int_{\R^N}V\mu(0) \, dx\\ & =\int_{\R^N} \nabla v(T)\nabla \mu(T)\, dx+\frac{1}{2}\int_{\R^N}|\nabla v(T)|^2\mu(T)\, dx+\int_{\R^N}F(\mu(T))\, dx -\int_{\R^N}V\mu(T) \, dx \\
& \leq -\int_{\R^N} \Delta v(T) \mu(T)\, dx+\frac{1}{2}\int_{\R^N}|\nabla v(T)|^2\mu(T)\, dx+\frac{\sigma}{\alpha+1}\int_{\R^N}\mu^{\alpha+1}(T)\, dx-\int_{\R^N}V\mu(T) \, dx
\end{align*}
Since the second and third term in the left-hand side of the previous equality are positive and since the last term in the right-hand side is negative, we get
$$
\int_{\R^N}\nabla v(0)\nabla \mu(0)\, dx\leq ||\Delta v(T)||_{L^\infty(\R^N)}+\frac{1}{2}||\nabla v(T)||^2_{L^\infty(\R^N)}+\frac{\sigma}{\alpha +1}\int_{\R^N}\mu^{\alpha+1}(T)\, dx+2||V||_{L^\infty(\R^N)}
$$
and back to \eqref{A} we obtain
\begin{multline}\label{estdersec}
\int_0^T\int_{\R^N}|D^2v|^2\mu \, dx dt\leq \frac{4\sigma\alpha}{(\alpha+1)^2}\int_0^T\int_{\R^N}|\nabla \mu^{\frac{\alpha+1}{2}}|^2\, dx dt+2||\Delta v(T)||_{L^\infty(\R^N)} +\frac{1}{2}||\nabla v(T)||^2_{L^\infty(\R^N)}+\\\frac{\sigma}{\alpha+1}\int_{\R^N}\mu^{\alpha+1}(T)\, dx+T||\Delta V||_{L^\infty(\R^N)}+2||V||_{L^\infty(\R^N)}.
\end{multline}

{\bf Step 4: conclusion}. Finally, we adopt the following notation for simplicity:
\[
\begin{array}{ll}
\displaystyle A:= \sup_{t \in [0,T]}\int_{\R^N}\mu^{\alpha+1}\, dx, & \displaystyle B:=
\int_0^t\int_{\R^N}|\nabla \mu^{\frac{\alpha+1}{2}}|^2\, dxdt, \\
\displaystyle D=\int_0^T\int_{\R^N}\mu^{2\alpha+1}\, dxdt, &
 \displaystyle F:=2||\Delta v(T)||_{L^\infty(\R^N)} +\frac{1}{2}||\nabla v(T)||^2_{L^\infty(\R^N)}+2||V||_{L^\infty(\R^N)}.
 \end{array}
\]
By plugging \eqref{estdersec} into \eqref{B}, we have 
\[
D^{\frac4q} + (A+B)^2 \le c\left[\int_{\R^N}\eta m_0^{\alpha+1}\, dx + \sigma D\left(\frac 1{\alpha+1} A + \frac{4}{(\alpha+1)^2} B\right) + F D +T||\Delta V||_{L^\infty(\R^N)}D+ 1  \right],
\]
and Young's inequality implies
\[
D^{\frac4q}  \le c\left[\int_{\R^N}m_0^{\alpha+1}(0)\, dx  + F D +T||\Delta V||_{L^\infty(\R^N)}D+ 1  \right] + \hat c \sigma^2 D^2
\]
for some $\hat c > 0$ depending on $c$. Since $\frac4q > 1$, by a further application of Young's inequality we can absorb the term involving $F D$ in the right-hand side into the left-hand side, to get the conclusion. Note that the constant $C$ in the statement of the theorem depends on $||m_0||_{L^{\alpha+1}(\R^N)}$, and on $c, F$, hence on $N, \alpha, ||u_T||_{C^2(\R^N)}$.
\end{proof}

Now we prove Theorem \ref{thm:exvmu}. Our aim is to apply the fixed-point Theorem \ref{prop:fixpoint}. First, we exploit the presence of a quadratic Hamiltonian to employ the standard Hopf-Cole change of variables $w = e^{-u/2}$. In the unknowns $w,m$ the MFG a system reads as system of coupled linear equations in divergence form
\[
\begin{cases}
	-w_t-\Delta w=\big(f(m)-V(x)\big)w & \mbox{ in } \R^N \times (0,T),\\
	m_t -\Delta m + 2\di\big(\frac{\nabla w}w m\big)=0 & \mbox{ in } \R^N \times (0,T),\\
	m(0)=m_0, \quad w(T)=e^{-u_T / 2} & \mbox{ on } \R^N.
\end{cases}
\]
We are going to prove the existence of a smooth couple $w,m$ solving this linear system, which yields immediately a solution to \eqref{2}.
\begin{proof}[Proof of Theorem \ref{thm:exvmu}]
	Let $X=L^{2\alpha+1}((0,T) \times \R^N)$, endowed with the topology of the strong convergence, and let $U = U_M$ be the open subset
	$$
U=\left\{m \in X\, : \, 	\int_0^T\int_{\R^N}m^{2\alpha+1}\, dxdt <M\right\} .
	$$
	The constant $M>0$ will be fixed later on.
	We define the operator $\mathcal{F}: \overline U \mapsto C$ as
	$$
	\mathcal{F}(\textit{m})=\mu, \quad m \in \overline U
	$$
	where $(w, \mu)$ is the solution of the system
	\begin{equation}\label{fxpo}
\begin{cases}
	-w_t-\Delta w=(f(m) -V(x)) w & \mbox{ in } \R^N \times (0,T),\\
	\mu_t -\Delta \mu + 2\di\big(\frac{\nabla w}w \mu\big)=0 & \mbox{ in } \R^N \times (0,T),\\
	\mu(0)=m_0, \quad w(T)=e^{-u_T / 2} & \mbox{ on } \R^N.
\end{cases}
\end{equation}
	
	{\bf $\mathcal{F}$ is well-posed and compact. } 
	For any $m \in \overline U$, since $\alpha < \frac{2}{N-2}$, then $m^\alpha \in L^p(\R^N \times (0,T))$ for $ p=\frac{2\alpha+1}{\alpha} > \frac{N+2}2$, which implies
	\[
	f(m) \in L^p(\R^N \times (0,T)), \qquad p=\frac{2\alpha+1}{\alpha} > \frac{N+2}2.
	\]
	Since $w(T) \in L^\infty(\R^N)$, the existence of a weak solution $w \in V_{2, \rm loc}(\R^N \times (0,T)) \cap L^\infty(\R^N \times (0,T))$ is standard (see, e.g. \cite[Chapter 3]{LadyP}). For the sake of self-containedness, we remark that  the bound  in $L^\infty(\R^N) $ of $w$ can be inferred  from Lemma \ref{lem:z} with $f_1=-V, f_2=f(m), g_1\equiv g_2\equiv h \equiv \eta \equiv 0$. Moreover, by the comparison principle (recall that $f \ge 0$), 
	$$
	w \ge \underline w := e^{-T\max_{\R^N \times[0,T]}V} \min_{\R^N} e^{-u_T / 2} > 0, 
	$$ 
	and $\nabla w \in L^{2p}(\R^N \times (0,T))$ by Lemma \ref{lem:z} with $g_1=Vw$ and $g_2=f(m) w$. Note that $\mu(0) \in L^1(\R^N) \cap L^\infty (\R^N)$ so it belongs in particular to $L^2 (\R^N)$, and since $\frac{\nabla w}w \in  L^{2p}(\R^N \times (0,T))$, $2p > N+2$, the existence of a solution $\mu \in V_{2}(\R^N \times (0,T)) \cap L^\infty(\R^N \times (0,T))$, which is also H\"older continuous, is again standard. Moreover, $\mu(t) \in L^1(\R^N)$ is bounded as $t$ varies in $[0,T]$ (see for example \cite{BiCoCrSpi}), and in particular $\int_{\R^N} \mu(t)\, dx= 1$ for all $t$ by Lemma \ref{lem:mombounds}. By interpolation, $\mu \in L^{2\alpha+1}(\R^N \times (0,T))$, so that $\mathcal{F}(m) \in X$.
	
	It will be useful to note that the estimates mentioned above imply that there exist $K, \theta >0$ depending on $M, N, T, \sigma, \alpha, ||u_T||_{C^2(\R^N)}, ||m_0||_{L^{\infty}(\R^N)}, ||m_0||_{L^{1}(\R^N)}, ||xm_0||_{L^1(\R^N)}$,  such that for any $m \in \overline U$
	\[
	\|w\|_{L^\infty(\R^N \times (0,T))}, \ \Big\|\frac{\nabla w}w \Big\|_{L^{2p}(\R^N \times (0,T))}, \ \|\mu\|_{L^\infty((0,T); L^q(\R^N))} \le K \qquad \text{for all $q \in [1,+\infty]$},
	\]
	and
	\[
	\sup_{t \in [0,T]} \int_{\R^N} |x|\mu(x,t)  dx, \ \sup_{z \in \R^N} \|\mu\|_{C^{\theta, \theta/2}(B_1(z) \times [0,T])} \le K.
	\]
	The latter inequalities are crucial to deduce compactness of $\mathcal{F}(\overline  U)$ in $X$. Indeed, let $m_n$ be a sequence in $\bar U$. We have, for $R > 0$,
	$$
	\sup_n \sup_{t \in [0,T]} \int_{|x|\geq R} \mu_n(x,t) \, dx\leq \frac{K}{R}.
	$$
	Since $m_n$ is bounded in $L^\infty(\R^N \times (0,T))$, by interpolation between $L^1$ and $L^\infty$ we have that for all $\eps > 0$ there exists $R$ large such that
	\begin{equation}\label{boundmun}
	\sup_n \|\mu_n\|_{L^\infty((0,T); L^{2\alpha+1}(\R^N \setminus B_R))} \le \eps.
	\end{equation}
	On the other hand, the sequence $\mu_n$ is bounded in $C^{\theta, \theta/2}(B_R)$. Hence,  by the Ascoli-Arzel\`a Theorem we can extract a subsequence $\mu_{n_k}$ uniformly converging to $\mu$ on $B_R \times [0,T]$. Note that \eqref{boundmun} holds for $\mu$ as well by the same considerations as above. Therefore, for any $n_k$  large enough,
	\[
	\|\mu_{n_k}-\mu\|_{L^\infty((0,T); L^{2\alpha+1}(B_R))} \le \eps.
	\]
	A standard diagonalization arguments yields convergence of subsequences in $L^\infty((0,T); L^{2\alpha+1}(\R^N))$, and therefore in $L^{2\alpha+1}(\R^N \times (0,T))$.

	{\bf $\mathcal{F}$ is continuous. } This is a standard stability argument. Pick any sequence $\{m_n\} \subset \overline U$ converging to $m$ in $X$. Then, $m_n^\alpha \to m^\alpha$ in $L^p(\R^N \times (0,T))$, which by \eqref{eq:ipofex} implies $f(m_n) \to f(m)$ in $L^p(\R^N \times (0,T))$. Denoting by $(m_n, w_n, \mu_n)$ the triple solving \eqref{fxpo}, $\tilde{w}=w-w_n$ satisfies
	$$
	-\tilde{w}_t-\Delta \tilde{w}= w(f(m)-f(m_n))+ (f(m_n)-V) \tilde{w}, \quad \tilde{w}(T)=0.
	$$
	Then we can apply Lemma \ref{lem:z} with $z= \tilde w, f_1=-V,  f_2= f(m_n), g_1\equiv 0, g_2= w(f(m)-f(m_n))$, and $h\equiv \eta\equiv 0$ to get
	$$
	||w-w_n||_{L^\infty(\R^N \times (0,T))}\leq  C  ||w(f(m)-f(m_n))||_{L^p(\R^N\times (0,T))} \le C K ||f(m)-f(m_n)||_{L^p(\R^N\times (0,T))} ,
	$$
	where $C$ depends on $K, N, T, p, \sigma$, and again by Lemma \ref{lem:z} with $g_1\equiv 0, g_2= w(f(m)-f(m_n))+ (f(m_n) - V) \tilde w$,
	\begin{multline*}
	||\nabla w- \nabla w_n||_{L^{2p}(\R^N \times (0,T))} \leq  C ||w(f(m)-f(m_n))+(f(m_n)-V) \tilde{w}||_{L^p(\R^N\times (0,T))}\\
	\leq C (K||f(m)-f(m_n)||_{L^p(\R^N\times (0,T))}+ (\sigma M^{\frac1p}+ \|V\|_\infty) ||w-w_n||_{L^\infty(\R^N \times (0,T))}).
	\end{multline*}
	Therefore, $w_n \to w$ in $L^\infty(\R^N \times (0,T))$ and $\nabla w_n \to \nabla w$ in $L^{2p}(\R^N \times (0,T))$. Note that since $w,w_n \ge \underline w > 0$ and
	$$
	\frac{\nabla w_n}{w_n}-\frac{\nabla w}{w}=\frac{1}{w}(\nabla w_n-\nabla w)+\nabla w_n\left(\frac{w-w_n}{w_nw}\right),
	$$
	then we also have $\frac{\nabla w_n}{w_n} \to \frac{\nabla w}{w}$ in $L^{2p}(\R^N \times (0,T))$. Setting $\tilde \mu = \mu - \mu_n$, it satisfies
	$$
	\tilde \mu_t -\Delta \tilde \mu = -2\di\big(\frac{\nabla w}w \tilde \mu\big) - 2\di\big(\big(\frac{\nabla w}{w} -\frac{\nabla w_n}{w_n}) \mu_n\big), \quad \tilde{\mu}(0)=0,
	$$
	so Lemma \ref{lem:z} with $f_1 \equiv f_2\equiv g_1\equiv g_2\equiv 0, h=\frac{\nabla w}{w}$ and $\eta=\left(\frac{\nabla w}{w} -\frac{\nabla w_n}{w_n}\right) \mu_n$ implies
	$$
	||\mu - \mu_n||_{L^\infty(\R^N \times (0,T))}\leq  C \left\|\big(\frac{\nabla w}{w} -\frac{\nabla w_n}{w_n}) \mu_n\right\|_{L^{2p}(\R^N\times (0,T))}\leq  CK \left\|\frac{\nabla w}{w} -\frac{\nabla w_n}{w_n}\right\|_{L^{2p}(\R^N\times (0,T))},
	$$
	where $C$ depends on $K, N, T, p$, which shows that $\mu_n \to \mu$ in $L^\infty(\R^N \times (0,T))$. Since $\mu_n, \mu$ are equi-bounded in $L^\infty(0,T;L^1(\R^N))$ we get that $\mathcal{F}(m_n) = \mu_n \to \mathcal{F}(m) = \mu$ in $X$.

	{\bf $\mathcal{F}$ satisfies: $\mu=\eta \mathcal{F}(\mu)$ for some $\eta \in (0,1)  \Rightarrow \mu \notin \partial U$. } In other words, we need to prove that if $(w,\mu)$ satisfies
	\[
\begin{cases}
	-w_t-\Delta w=(f(\mu) - V)w & \mbox{ in } \R^N \times (0,T),\\
	\mu_t -\Delta \mu + 2\di\big(\frac{\nabla w}w \mu\big)=0 & \mbox{ in } \R^N \times (0,T),\\
	\mu(0)=\eta m_0, \quad w(T)=e^{-u_T / 2} & \mbox{ on } \R^N.
\end{cases}
\]
for some $\eta \in (0,1)$, then $\int_0^T\int_{\R^N}m^{2\alpha+1} \, dx dt \neq M$ (with $v=-2\log w$). To this aim, we apply Theorem \ref{thm:estex}. Note that Theorem \ref{thm:estex} requires $(w,\mu)$ to be a classical solution, so we need first to set up a (standard) bootstrap procedure involving parabolic Schauder estimates. Starting from the fact that $\mu=\eta \mathcal{F}(\mu)$ belongs to $C^{\theta, \theta/2}$ (locally), then $w$ belongs to $C^{2+\theta, 1+\theta/2}$ and it is a classical solution of the first equation. Computing the divergence term in the second equation shows that $\mu$ solves a linear equation with coefficients in $C^{\theta, \theta/2}$, so $\mu$ itself belongs to $C^{2+\theta, 1+\theta/2}$. Going back to the equation for $w$ and iterating the procedure, one can reach any desired regularity of $(w,m)$ that is compatible with the regularity of $u_T$ and $m_0$.

Theorem \ref{thm:estex} gives  $D^\delta\leq C(\sigma^2D^2+T||\Delta V||_{L^\infty(\R^N)}D+1)$ for some $\delta>1$, which by the Young's inequality  implies $D^\delta\leq C(\sigma^2D^2+1)$ where $C$ depends now also on $T||\Delta V||_{L^\infty(\R^N)}$.
Then, setting $Y = \left(\int_0^T\int_{\R^N}m^{2\alpha+1}\right)^2$, $Y^{\delta/2} \le C(\sigma^2 Y + 1)$ for some $\delta/2 < 1$, that is
\begin{equation}\label{superl}
\frac{Y^{\frac\delta2}}C - \sigma^2 Y \le 1.
\end{equation}
Since $Y \mapsto \frac{Y^{\frac\delta2}}C - \sigma^2 Y$ is concave, vanishing at $Y=0$ and unbounded from below as $Y \to \infty$, it achieves a unique positive maximum point $\frac{(Y^*)^{\frac\delta2}}C - \sigma^2 Y^*$. If that maximum is bigger than one, that is when $$\sigma \le \sigma_0 = \left( \frac \delta {2C}\right)^{\frac2\delta}\left(\frac2\delta-1 \right)^{\frac{2-\delta}2}, $$ then \eqref{superl} implies
\[
Y < Y^* \quad \text{or} \quad Y > Y^*,
\]
that is $\left(\int_0^T\int_{\R^N}m^{2\alpha+1}\right)^2 \neq Y^*$. Note that $\sigma_0, Y^*$ depend on $C$ and $ \delta$, that is, on $N, \alpha,$ $||u_T||_{C^2(\R^N)}$, and $||m_0||_{L^{\alpha+1}(\R^N)}$.  It is now clear that setting $M= \sqrt{Y^*}$ yields the desired property.

\smallskip

We are now in the position to apply Theorem \ref{prop:fixpoint} (with $C=X$, $p=0$), that gives the existence of a fixed point $\mu = \mathcal F(\mu)$, i.e. a couple $(m,w)$ solving the linear system in the classical sense. A classical solution to the MFG system can be then recovered via the reverse change of variables $u=-2\log w$.
\end{proof}

\begin{rem}[Thoughts on the long-time behavior of solutions]\label{longtime} \upshape In case $V \equiv 0$, Theorem \ref{thm:exvmu} states, for any fixed coupling $f$ sufficiently ``small'', the existence of solutions to the MFG system for all $T > 0$. This opens the way to the study of the long-time behavior of $m$ and $u$. Though this analysis is beyond the scopes of this paper, the proof of 
Theorem \ref{thm:exvmu} suggests that $m$ should ``vanish'' as $T \to \infty$. Indeed, $m$ satisfies
$$
\int_0^T\int_{\R^N}m^{2\alpha+1}(x,t)\, dxdt <M
$$
for some $M > 0$ that does not depend on $T$. Therefore,
$$
\int_0^1\int_{\R^N}m^{2\alpha+1}\big(x,\frac tT\big)\, dxdt \to 0 \quad \text{as $T \to \infty$},
$$
indicating that $m$ dissipates as $T \to \infty$ (and hence there is no ergodic behavior). Note also that stability of the $L^{2\alpha+1}$-norm of $m$ as $T \to \infty$ is sufficient to produce global bounds on $\|m\|_\infty$  and $\|Du\|_\infty$; arguing as in \cite{CirPo}, a possibly stronger ``smallness'' condition on the coupling (i.e. a smallness condition on $\sigma$) guarantees then {\it uniqueness} of the couple $(u,m)$ for all $T$.

\end{rem}

\appendix

\section{Some useful estimates}

We collect in this appendix some results that are used throughout the paper. We begin with some facts that are useful to integrate by parts on $\R^N$.

\begin{lem}\label{lem:lem32time}
	Let $t_1, t_2 \in [0,T]$ and $h \in L^1( \R^N \times [t_1,t_2])$. There there exists a sequence $R_n \to \infty$ such that 
	$$
	R_n \int_{t_1}^{t_2}\int_{\partial B_{R_n}}|h(x,t)| \, dxdt \to 0 \quad \mbox{ as } n \to \infty.
	$$
	\end{lem}
	\begin{proof}
		The proof follows the same arguments of the proof of Lemma $3.2$ of \cite{CirCPDE16}. We repeat it for completeness. By the coarea formula, it holds true that
		\begin{equation}\label{eqn:lemcoarea}
		\int_{t_1}^{t_2}\int_{\R^N} |h(x,t)|\, dx dt=\int_{t_1}^{t_2} \int_{0}^\infty\int_{\partial B_R} |h(x,t)|\,dx dR dt=\int_0^\infty \int_{t_1}^{t_2} \int_{\partial B_R} |h(x,t)|\,dx dtdR<\infty.
		\end{equation}
		If, by contradiction
		$$
		\liminf_{R \to \infty}R\int_{t_1}^{t_2} \int_{\partial B_R} |h(x,t)|\, dxdt =\alpha>0
		$$
		then
		$$
		R \mapsto \int_{t_1}^{t_2} \int_{\partial B_R} |h(x,t)|\, dxdt
		$$
		would not be in $L^1(0,\infty)$, which is not compatible with \eqref{eqn:lemcoarea}.
		\end{proof}

\begin{lem}\label{lem:intp} Suppose that $Y, f \in C^1(\R^N \times [t_1,t_2])$, and $$f{\rm div} Y, \ \ Y \cdot \nabla f, \ \ \frac{f|Y|}{1+|x|} \ \ \in L^1(\R^N \times [t_1,t_2]).$$ Then, the following equality holds:
\[
\int_{t_1}^{t_2}\int_{\R^N} {\rm div} Y \, f \, dx dt=- \int_{t_1}^{t_2}\int_{\R^N} Y \cdot \nabla f \, dx dt.
\]
\end{lem}

Note that an analogous formula holds on $\R^N$ with the Lebesgue measure $dx$.

\begin{proof} The integration by parts formula holds on any bounded domain of the form $B_R \times (t_1,t_2)$, i.e.
\[
\int_{t_1}^{t_2}\int_{B_R} {\rm div} Y \, f \, dx dt=- \int_{t_1}^{t_2}\int_{B_R} Y \cdot \nabla f \, dx dt + \int_{t_1}^{t_2}\int_{\partial B_R} f Y \cdot \nu \, dxdt.
\]
By the previous Lemma \ref{lem:lem32time} (applied with $h = \frac{f|Y|}{1+|x|}$), there exists a sequence $R=R_n \to \infty$ such that 
	$$
	 \int_{t_1}^{t_2}\int_{\partial B_{R_n}}f |Y| \, dxdt \to 0 \quad \mbox{ as } R_n \to \infty,
	$$
so that it suffices to pass to the limit $R_n \to \infty$ in the first equality to obtain the assertion. 

\end{proof}

The following lemma is a (standard) regularity result for linear parabolic equations. Since we have not been able to find it in this precise form in the literature, we sketch its proof for the reader's convenience.
\begin{lem}\label{lem:z}
	
	Suppose that $z$ is a (bounded) weak solution of the following linear equation
	$$
	z_t-\Delta z=(f_1+f_2)z+g_1+g_2+\di(hz)+\di(\eta), 
	$$
	where $f_1, g_1 \in L^\infty( \R^N \times (0,T))$, $f_2,g_2\in L^p( \R^N \times (0,T))$ and $h, \eta \in L^{2p}(\R^N\times (0,T)) $, for some $$p> \frac{N+2}{2}.$$
Then  
	\begin{equation}\label{boundz}
	||z||_{L^\infty(\R^N \times (0,T))}\leq  C \left(||g_1||_{L^\infty(\R^N\times (0,T))}+||g_2||_{L^p(\R^N\times (0,T))}+||\eta||_{L^{2p}(\R^N\times (0,T))}+||z(0)||_{L^\infty(\R^N)}\right),
	\end{equation}
	where $C$ depends on $f_1, f_2, h, N, T, p$, and remains bounded for bounded values of $||f_1||_{L^\infty(\R^N\times (0,T))}$, $ ||f_2||_{L^p(\R^N\times (0,T))}$, $||h||_{L^{2p}(\R^N\times (0,T))}$.
	
Suppose $f_1\equiv f_2 \equiv h \equiv \eta \equiv 0$. Then
	\begin{equation}\label{boundDz}
	||\nabla z||_{L^{2p}(\R^N \times (0,T))}\leq  C \left(||g_1||_{L^\infty(\R^N\times(0,T))}+||g_2||_{L^p(\R^N\times (0,T))}+||\nabla z(0)||_{L^{2p}(\R^N)}\right),
	\end{equation}
	where $C$ depends on $N, T, p$.
	\end{lem}

\begin{rem}\label{lem:zfor}
	Note that a similar result applies to the backward equation
	$$
	-z_t-\Delta z=f_1z+f_2z+g_1+g_2+\di(hz)+\di(\eta),
	$$
	More specifically, under the same assumptions of Lemma \ref{lem:z}, we have the same estimates with $z(T)$ in place of $z(0)$.
\end{rem}

\begin{proof}
 
	We start with \eqref{boundz}. For simplicity of exposition, suppose first $f_1\equiv g_1\equiv 0 $ and set $f_2:=f, g_2:=g$. Denote by $F=fz+g+\mbox{div}(hz)+\mbox{div}(\eta)$. 
	We use the Duhamel representation formula, that is
	$$
	z(\cdot, t)=z(0)*G(\cdot, t)+\int_0^tG(\cdot, t-s)*F(\cdot, s)\, ds,
	$$
	where we denote by $*$ the convolution (in space) and $G(x,t)=\frac{1}{(4\pi t)^{\frac{N}{2}}}e^{-\frac{|x|^2}{4t}}$ is the heat kernel.
	By the definition of $F$ we have
	$$
	\int_0^tG(\cdot, t-s)*F(\cdot, s)\, ds=\int_0^t[G(\cdot, t-s)*(fz+g)(\cdot, s)+\nabla G(\cdot, t-s)* (hz + \eta)(\cdot, s) ]\, ds.
	$$
	Using Young's inequality with
	\begin{equation}\label{qq'}
	\frac{1}{p}+\frac{1}{q}=1, \quad \frac{1}{2p}+\frac{1}{r}=1
	\end{equation}
	we get
	\begin{align}\label{zinfty}
	||z(t)||_{L^\infty(\R^N)} \leq & ||z(0)||_{L^\infty(\R^N)}+\int_0^t ||G(t-s)||_{L^q(\R^N)}[||fz(s)||_{L^p(\R^N)}+||g(s)||_{L^p(\R^N)}]\, ds \\& + \int_0^t||\nabla G(t-s)||_{L^{r}(\R^N)}[||hz(s)||_{L^{2p}(\R^N)}+||\eta(s)||_{L^{2p}(\R^N)}]\, ds\nonumber.
	\end{align}
	Denoting by $S= S_\tau = \sup_{t \in [0,\tau]} ||z(t)||_{L^\infty(\R^N)}$, for $\tau > 0$, the last inequality reads for all $t \in [0, \tau]$
	\begin{align}
	||z(t)||_{L^\infty(\R^N)} \leq & ||z(0)||_{L^\infty(\R^N)}+\int_0^t ||G(t-s)||_{L^q(\R^N)}[S ||f(s)||_{L^p(\R^N)}+||g(s)||_{L^p(\R^N)}]\, ds \\& + \int_0^t||\nabla G(t-s)||_{L^{r}(\R^N)}[S||h(s)||_{L^{2p}(\R^N)}+||\eta(s)||_{L^{2p}(\R^N)}]\, ds \nonumber.
	\end{align}
	By the H\"older inequality,
	\begin{align}\label{zinfty2}
	||z(t)||_{L^\infty(\R^N)} \leq & ||z(0)||_{L^\infty(\R^N)}+ ||G||_{L^q(\R^N \times (0,t) )}[S ||f||_{L^p(\R^N \times (0,t))}+||g||_{L^p(\R^N \times (0,t))}] \\& +  ||\nabla G||_{L^{r}(\R^N \times (0,t))}[S||h||_{L^{2p}(\R^N \times (0,t))}+||\eta||_{L^{2p}(\R^N \times (0,t))}]  \nonumber.
	\end{align}
	It is now standard to verify (e.g. using spherical coordinates) that, since $p> \frac{N+2}{2}$,
	\begin{equation}\label{nablaG}
	||G||_{L^q(\R^N \times (0,t) )} =C_1 t^{\frac{N}{2q}-\frac N 2+\frac 1 q} =: C_1 t^{\beta_1}, \qquad ||\nabla G||_{L^r(\R^N \times (0,t) )} =C_2 t^{\frac{N}{2r}-\frac {N+1} 2+\frac 1 r}=: C_2 t^{\beta_2},
	\end{equation}
	where $\beta_1, \beta_2, C_1, C_2 > 0$ depend on $N$ and $q, r$ (and therefore on $p$).
	By plugging the above equalities into \eqref{zinfty2}, we get
	\begin{align}\label{zinfty3}
	S \leq & ||z(0)||_{L^\infty(\R^N)}+ S[ C_1 \tau^{\beta_1}  ||f||_{L^p(\R^N \times (0,\tau))} + C_2 \tau^{\beta_2} ||h||_{L^{2p}(\R^N \times (0,\tau))}]  \\&  + C_1 \tau^{\beta_1}  ||g||_{L^p(\R^N \times (0,\tau))}+C_2 \tau^{\beta_2} ||\eta||_{L^{2p}(\R^N \times (0,\tau))} \nonumber.
	\end{align}
	Pick now $\tau = \bar \tau$ small so that 
	\[
	 C_1 \tau^{\beta_1}  ||f||_{L^p(\R^N \times (0,T))} + C_2 \tau^{\beta_2} ||h||_{L^{2p}(\R^N \times (0,T))} \le \frac12,
	 \]
	 so that
	 \[
	 \sup_{t \in [0,\tau]} ||z(t)||_{L^\infty(\R^N)} \le 2||z(0)||_{L^\infty(\R^N)} + 2C_1 {\bar \tau}^{\beta_1}  ||g||_{L^p(\R^N \times (0,\tau))}+2C_2 {\bar \tau}^{\beta_2} ||\eta||_{L^{2p}(\R^N \times (0,\tau))},
	 \]
	which is the desired estimate \eqref{boundz} if $T \le \bar \tau$. If $T > \bar \tau$, it is sufficient to iterate the estimate $n$ times, where $n$ is the integer part of $\frac{T}{\bar \tau}$, to get the result.
	
	Finally, if $f_1 \nequiv 0$ or $g_1\nequiv 0$, then analogous argument applies, by using in \eqref{zinfty} the Young's inequality with $p=\infty, q=1, r=\infty$ for the term $f_1z$ and $g_1$.
	
	\smallskip
	
	Inequality \eqref{boundDz} can be obtained similarly. Again we prove it in the case $g_1\equiv0$ (and denoting $g_2:=g$), the proof in the general case being easily adaptable using Young's inequality with $\frac{1}{2p}=\frac{1}{r}-1$ with $r\leq 2p$ for the term $g_1**\nabla G$ below.  It is convenient to rewrite the Duhamel representation formula as follows (recall that $f_1\equiv f_2 \equiv h \equiv \eta \equiv 0$),
	$$
	\nabla z( t)=\nabla z(0) *G(t)+ g**\nabla G(t),
	$$
	where by $**$ we denote the convolution in space-time.
	By the Young inequality with $r$ such that (as before)
	$$
	\frac{1}{2p}=\frac{1}{r}+\frac{1}{p}-1,
	$$
	we get
	\begin{align}\label{nablaw}
	||\nabla z||_{L^{2p}(\R^N \times (0,T))}\nonumber \leq & ||\nabla z(0)*G||_{L^{2p}(\R^N \times (0,T))}+\left|\left|g**\nabla G\right|\right|_{L^{2p}(\R^N\times (0,T)}\\ \leq & T^{\frac1{2p}}||\nabla z(0)||_{L^{2p}(\R^N)}+||g||_{L^{p}(\R^N \times (0,T))}||\nabla G||_{L^r(\R^N \times (0,T))}.
	\end{align} 
	We plug \eqref{nablaG} into \eqref{nablaw} and we get
	$$
	||\nabla z||_{L^{2p}(\R^N \times (0,T))} \leq T^{\frac1{2p}}||\nabla z(0)||_{L^{2p}(\R^N)}+C_2 T^{\beta_2} ||g||_{L^{p}(\R^N \times (0,T))},
	$$
	which is the desired estimate.
	\end{proof}

Finally, we recall some facts on distributional solutions (and their moments) to Fokker-Planck equations.

\begin{lem}\label{lem:mombounds} Let $\{\mu(x, t) dx\}_{t \in [0,T]}$ be a family of probability measures on $\R^N$, and $b(x,t)$ be a measurable vector field which is $\mu(x,t) dx dt$-integrable on $\R^N \times (0,T)$. Let $\mu$ be a distributional solution of the Fokker-Planck equation with drift $b$ (and diffusion), i.e.
\begin{equation}\label{fkpd}
\int_{\R^N} \zeta(x) \mu(x,t) \, dx = \int_{\R^N} \zeta(x) \mu(x,0) \, dx + \int_0^t \int_{\R^N} \big(\Delta \zeta(x) + b(x,t) \cdot \nabla \zeta(x) \big)\mu(x,t) \, dxdt
\end{equation}
for all $t \in [0,T]$, $\zeta \in C^{\infty}_0(\R^N)$. We have
\begin{itemize}
\item[{\it (i)}]  If $|x|\mu(0) \in L^1(\R^N)$, then $|x|\mu(t) \in L^1(\R^N)$ for all $t$, and $$\sup_{t \in [0,T]} \int_{\R^N} |x| \mu(t,x) \, dx \le C$$ for some $C > 0$ depending on $\int_{\R^N} |x| \mu(0) dx,  \int_0^T \int_{\R^N} |b| \mu dx dt, T$.
\item[{\it (ii)}] In addition, if $x^2 \mu(0) \in L^1(\R^N)$ and $|b \cdot x| \mu \in L^1(\R^N \times (0,T))$, then $x^2\mu(t) \in L^1(\R^N)$ for all $t$, and $$\int_{\R^N} x^2\mu(x,t) \, dx = \int_{\R^N} x^2\mu(x,0) \, dx + 2Nt \int_{\R^N} \mu(x,0)dx + \int_0^t \int_{\R^N} b(x,t) \cdot x \, \mu(x,t) \, dxdt. $$
\end{itemize}

\end{lem}

\begin{proof}

{\it (i)} and {\it (ii)} heuristically follow by using $|x|$ and $x^2$ respectively as test functions; since they do not belong to $C^{\infty}_0(\R^N)$, one has to implement (standard) approximation procedures. \\
To get {\it (i)}, it is sufficient to use \cite[Lemma 2.2]{DaPrato} with $\Psi \in C^\infty(\R^N)$, $\Psi(x) = |x|$ outside $B_1(0)$.\\
As for {\it (ii)}, note first that by \cite[Lemma 2.2]{DaPrato} (with $\Psi(x) = x^2$), $\sup_{t \in [0,T]} \int_{\R^N}x^2\mu(t)dx < \infty $. To obtain the identity, note first that \eqref{fkpd} remains true for all $\zeta \in C^{2}_b(\R^N)$ such that $\zeta$ is constant outside some ball. Then, for $R > 0$ take any $\varphi_R \in C^2(\R)$ such that $\varphi_R(r) = r$ if $r \in [0,R]$, $|\varphi_R'|, |\varphi_R''| \le 1$ and $\varphi_R(r) = \varphi_R(R)$ for all $r \ge R+1$. Then, using $\zeta(x) = \varphi_R(x^2)$ in \eqref{fkpd} reads
\begin{multline*}
\int_{\R^N}  \varphi_R(x^2) \mu(x,t) \, dx= \int_{\R^N}  \varphi_R(x^2) \mu(x,0) \, dx + 2 \int_0^T \int_{\R^N}\varphi'_R(x^2) \big(N + b(x,t) \cdot x \big)\mu(x,t) \, dxdt \\ +
4 \int_0^T \int_{\R^N} \varphi''_R(x^2) x^2 \mu(x,t) \, dxdt,
\end{multline*}
and the conclusion follows taking the limit $R \to \infty$.

\end{proof}

\end{document}